\documentclass{amsart}
\usepackage{amsthm}
\usepackage{amssymb}
\usepackage{tikz-cd}
\usepackage{colonequals}
\usepackage{dsfont}
\usepackage{enumerate}
\usepackage{eucal}

\usepackage{hyperref}
\hypersetup{%
  bookmarksnumbered=true,%
  colorlinks=true,%
  linkcolor=blue,%
  citecolor=blue,%
  filecolor=blue,%
  menucolor=blue,%
  urlcolor=blue,%
  bookmarksopen=true,%
  bookmarksdepth=2,%
  pageanchor=true}

\makeatletter
\@namedef{subjclassname@2020}{\textup{2020} Mathematics Subject Classification}
\makeatother

\numberwithin{equation}{section}

\swapnumbers

\theoremstyle{plain}
\newtheorem{theorem}[equation]{Theorem}
\newtheorem{proposition}[equation]{Proposition}
\newtheorem{lemma}[equation]{Lemma} 
\newtheorem{corollary}[equation]{Corollary}

\theoremstyle{definition}

\newtheorem{example}[equation]{Example}

\newtheorem{chunk}[equation]{}

\theoremstyle{remark}
\newtheorem{remark}[equation]{Remark}

\newcommand{\add}{\operatorname{add}}
\newcommand{\cat}[1]{\mathcal{#1}}
\newcommand{\Coker}{\operatorname{Coker}}

\newcommand{\bfD}{\mathbf{D}}
\newcommand{\dbcat}[1]{\mathbf{D}^{\mathrm{b}}(\operatorname{mod} #1)}
\newcommand{\dcat}[1]{\mathbf{D}(\operatorname{Mod} #1)}
\newcommand{\Ext}{\operatorname{Ext}}
\newcommand{\fibre}[2]{ {#1}_{k(#2)}}
\newcommand{\Gproj}{\operatorname{Gproj}}
\newcommand{\uGproj}{\operatorname{\underline{Gproj}}}
\newcommand{\GProj}{\operatorname{GProj}}
\newcommand{\uGProj}{\operatorname{\underline{GProj}}}
\newcommand{\hh}[1]{H^{*}(#1)}

\newcommand{\Hom}{\operatorname{Hom}}
\newcommand{\fHom}{\operatorname{\mathcal{H}\!\!\;\mathit{om}}}
\newcommand{\id}{\operatorname{id}}
\newcommand{\Inj}{\operatorname{Inj}}
\newcommand{\injdim}{\operatorname{inj{.}dim}}

\newcommand{\iso}{\xrightarrow{\raisebox{-.4ex}[0ex][0ex]{$\scriptstyle{\sim}$}}}

\newcommand{\longiso}{\xrightarrow{\ \raisebox{-.4ex}[0ex][0ex]{$\scriptstyle{\sim}$}\ }}

\newcommand{\Ker}{\operatorname{Ker}}
\newcommand{\bfK}{\mathbf{K}}

\newcommand{\KInj}[1]{\mathbf{K}(\Inj #1)}
\newcommand{\KMod}[1]{\mathbf {K}(\Mod #1)}
\newcommand{\kos}[2]{{#1}/\!\!/{#2}} 
\newcommand{\pres}{\mathbf{p}}
\newcommand{\Prj}{\operatorname{Proj}}
\newcommand{\KacProj}[1]{\mathbf{K}_{\mathrm{ac}}(\Prj #1)}
\newcommand{\KProj}[1]{\mathbf{K}(\Prj #1)}
\newcommand{\KProjc}[1]{\mathbf{K}^{\mathrm c}(\Prj #1)}
\newcommand{\kproj}[1]{\mathbf{K}_{\mathrm{proj}}(#1)}
\newcommand{\Loc}{\operatorname{Loc}}
\newcommand{\rmod}{\operatorname{mod}}
\newcommand{\Mod}{\operatorname{Mod}}
\newcommand{\one}{\mathds 1}
\newcommand{\projdim}{\operatorname{proj{.}dim}}
\newcommand{\RHom}{\operatorname{RHom}}
\newcommand{\Spec}{\operatorname{Spec}}

\newcommand{\stmod}{\operatorname{stmod}}

\newcommand{\supp}{\operatorname{supp}}
\newcommand{\Thick}{\operatorname{Thick}}
\newcommand{\tors}{\operatorname{tors}}

\newcommand{\lotimes}{\otimes^{\mathrm L}}
\newcommand{\op}[1]{{#1}^{\mathrm{op}}}
\newcommand{\lra}{\longrightarrow}
\newcommand{\xra}{\xrightarrow}
\newcommand{\bfj}{\mathbf{j}}
\newcommand{\bfq}{\mathbf{q}}
\newcommand{\bft}{\mathbf{t}}
  
 \newcommand{\bbZ}{\mathbb Z} 

\newcommand{\fm}{\mathfrak{m}} 
\newcommand{\fp}{\mathfrak{p}}
\newcommand{\fq}{\mathfrak{q}} 
\newcommand{\eps}{\varepsilon}
\newcommand{\gam}{\varGamma} 

\title[Fibrewise stratification]{Fibrewise stratification of group representations}
\author[Benson, Iyengar, Krause, and Pevtsova]{Dave Benson, Srikanth  B. Iyengar, Henning Krause, and Julia Pevtsova}

\address{Dave Benson \\ 
Institute of Mathematics\\ 
University of Aberdeen\\ 
King's College\\ 
Aberdeen AB24 3UE\\ 
Scotland U.K.}

\address{Srikanth B. Iyengar\\ 
Department of Mathematics\\
University of Utah\\ 
Salt Lake City, UT 84112\\ 
U.S.A.}

\address{Henning Krause\\ 
Fakult\"at f\"ur Mathematik\\ 
Universit\"at Bielefeld\\ 
33501 Bielefeld\\ 
Germany.}

\address{Julia Pevtsova\\ 
Department of Mathematics\\ 
University of Washington\\ 
Seattle, WA 98195\\ 
U.S.A.}

\begin{document}

\begin{abstract} 
  Given a finite cocommutative Hopf algebra $A$ over a commutative
  regular ring $R$, the  lattice of localising tensor ideals of the stable category of
  Gorenstein projective $A$-modules is described in terms of the corresponding lattices for the fibres of $A$ over the spectrum of $R$.
  Under certain natural conditions on the cohomology of $A$ over $R$, this yields a stratification of the stable category.
    These results apply when $A$ is the group algebra over $R$ of a finite
  group, and also when $A$ is the exterior algebra on a finite free
  $R$-module.
\end{abstract}



\date{\today}

\thanks{SBI was partly supported by NSF grant DMS-2001368, and JP was partly supported by NSF grant DMS-1901854 and the  Brian and Tiffinie Pang faculty fellowship.}

\maketitle

\setcounter{tocdepth}{1}
\tableofcontents

\section{Introduction}

Following the seminal work of Hopkins~\cite{Hopkins:1987a} and Neeman~\cite{Neeman:1992a} in stable homotopy theory and commutative algebra,  much attention has been paid in the past few decades to the problem of classifying the thick subcategories of finite dimensional representations over various families of algebras, and also of the localising subcategories of all representations.  In terms of the language and machinery developed in~\cite{Benson/Iyengar/Krause:2008a,Benson/Iyengar/Krause:2011a},  the goal is to prove stratification theorems. For example, in the case of modular representations of a finite group, the thick tensor ideal subcategories of the small stable module category were classified in~\cite{Benson/Carlson/Rickard:1996a}, while the tensor ideal localising subcategories of the large stable module category were classified in~\cite{Benson/Iyengar/Krause:2011b}. These results were generalised to cover all finite group schemes over fields in \cite{Benson/Iyengar/Krause/Pevtsova:2018a}.
 
In this paper we address the problem of change of coefficients, with a
focus on representations of group algebras of finite group schemes,
and in particular, of finite groups. Let  $A$ be the group algebra of
a finite group scheme over a commutative noetherian ring $R$; in other
words, $A$ is a finite cocommutative Hopf algebra over $R$.  For
example, $A$ could be the group algebra $RG$ of a finite group
$G$. The appropriate analogue of the stable category of finite
dimensional representations over a finite group, or group scheme, is
the singularity category of $A$, in the sense of
Buchweitz~\cite{Buchweitz:2021a} and Orlov~\cite{Orlov:2004a}. By a
result of Buchweitz, this is equivalent to the stable category of
Gorenstein projective $A$-modules:
\[
\uGproj A\longiso\bfD_{\mathrm{sg}}(A)\colonequals\dbcat A/\bfD^{\mathrm{perf}}(A)\,.
\]
An $A$-module is said to be Gorenstein projective if it occurs as a syzygy in an acyclic complex of projective $A$-modules. The Gorenstein projective modules in $\rmod A$  form a Frobenius category, with projective-injective objects the  projective $A$-modules; $\uGproj A$ is the corresponding stable category. We have also to consider all Gorenstein projectives, not only the finitely generated ones, and the corresponding stable category $\uGProj A$. The category $\uGProj A$ is triangulated and compactly generated; the subcategory of compact objects is equivalent to $\uGproj A$. 

From now on assume that $R$ is regular, for example, $R=\bbZ$. In this
case, a finitely generated $A$-module is Gorenstein projective
precisely when it is projective when viewed as an $R$-module. The same
holds also for infinitely generated $A$-modules when in addition
$\dim R$ is finite, that is to say, when $R$ has finite global
dimension; see Lemma~\ref{le:symmetric}. Since $A$ is a Hopf
algebra over $R$, the tensor product over $R$ induces on $\uGproj A$
the structure of a tensor-triangulated category and also on
$\uGProj A$. Moreover $\uGProj A$ is rigidly compactly generated. Our goal is to classify the thick tensor ideals of
$\uGproj A$ and the localising tensor ideals of $\uGProj A$.

One approach would be to extend methods developed in the case where $R$ is a field to cover more general coefficient rings. In this work,
we take a different tack, by viewing $A$ as a family of Hopf algebras parameterised by $\Spec R$, the Zariski spectrum of $R$. The fiber
over each point $\fp$ in $\Spec R$ is the finite dimensional Hopf algebra $\fibre A{\fp}\colonequals A\otimes_R k(\fp)$, where $k(\fp)$
is the residue field at $\fp$. The results of \cite{Benson/Iyengar/Krause/Pevtsova:2018a} apply to yield a stratification of $\uGProj  \fibre A{\fp}$ in terms of  the projective spectrum of the cohomology ring of $\fibre A{\fp}$. Then the task becomes one of `patching' these local stratifications to obtain a global stratification of $\uGProj A$ in terms of the projective spectrum of the cohomology ring of $A$.

There are two aspects to this task: one representation theoretic and
the other purely cohomological. The former is completely solved by the
result below that can be viewed as a fibrewise criterion for detecting
membership in localising tensor ideals. We deduce it from  Theorem~\ref{th:lg-kproj-ha} that deals with the full homotopy category of projective $A$-modules.

\begin{theorem}
\label{ith:lg-Gproj-ha}
Let $R$ be a regular ring, $A$ a finite cocommutative Hopf $R$-algebra, and $M,N$ Gorenstein projective $A$-modules. The  conditions below are equivalent.  
\begin{enumerate}[\quad\rm(1)]
\item
 $M\in  \Loc^{\otimes}(N)$ in $\uGProj A$;
\item
 $\fibre M\fp\in \Loc^{\otimes}(\fibre N{\fp})$ in $\uGProj \fibre A{\fp}$ for each $\fp\in \Spec R$. 
\end{enumerate}
\end{theorem}

The cohomological aspect concerns the relationship between the fibres of the cohomology algebra $S\colonequals \Ext^*_A(R,R)$ of $A$ and the cohomology algebra of the fibres of the $R$-algebra $A$. Namely, for each $\fp$ in $\Spec R$ there is natural map 
\[ 
\kappa_{\fp}\colon S\otimes_R k(\fp) \lra \Ext_{\fibre    A\fp}^*(k(\fp),k(\fp))
\]
of $k(\fp)$-algebras. The question is when this map induces a bijection on spectra.   Its import is clear from the next result.

\begin{theorem}
\label{ith:homeo=stratification}
Let $R$ be a regular ring and $A$ a finite cocommutative Hopf algebra over $R$ such that the $R$-algebra $S$ is finitely generated. If the map $\kappa_\fp$ induces a bijection on spectra for each $\fp$ in $\Spec R$, then the tensor-triangulated category
$\uGProj A$ is stratified by the action of $S$, and the support of
$\uGProj A$ is $\Prj S$.
\end{theorem}

This result is proved at the end of  Section~\ref{se:hopf-algebras},
where we also recall what it means for a tensor-triangulated to be
stratified by an action of a ring. For now it suffices to record that
the result above, when it applies, yields the sought after
classification of localising tensor ideals of $\uGProj A$ and of the
thick tensor ideals of $\uGproj A$. 

As to the hypothesis on $A$ in the theorem above: We expect that
finite generation holds for all finite cocommutative Hopf algebras. We
do not know if the hypothesis on $\kappa^a_\fp$ does as well; we
rather suspect that it does not. It does hold for the group algebra
$A=RG$ over a finite group. This  is a result of Benson and Habegger~\cite{Benson/Habegger:1987a}. The proof there is lacking detail, so we provide a full proof here in greater generality; see Theorem~\ref{th:f-iso-RG}, and also the recent work of Lau~\cite[Section~7]{Lau:Bsp}.   Putting this together with Theorem~\ref{ith:homeo=stratification} yields the following stratification result.

\begin{theorem}
\label{ith:stratification-uRG}
With $G$ a finite group and $R$ a regular ring, the
tensor-triangulated category $\uGProj RG$ is stratified by the action of the group cohomology ring $H^*(G,R)$.
\end{theorem}

In fact we prove this stratification result for the slightly bigger
compactly generated tensor-triangulated category $\KProj {RG}$
consisting of complexes of projective $RG$-modules up to homotopy. The
subcategory of compact objects identifies with the bounded derived
category $\dbcat {RG}$. The resulting classification of thick tensor ideals of $\dbcat {RG}$
has been obtained by Lau~\cite{Lau:Bsp} using  different
methods. Building on his work, and again using very different methods,
Barthel~\cite{Barthel:strat, Barthel:strat-regular} established a classification of 
localising tensor ideals of $RG$-modules that are projective as $R$-modules, which is closely related to Theorem~\ref{ith:stratification-uRG}.

Another case where Theorem~\ref{ith:homeo=stratification} applies is
when $A$ is an exterior algebra, over a finite free $R$-module,
regarded as a $\bbZ/2$-graded Hopf algebra; see
Example~\ref{ex:exterior}.

So far we have focussed on Hopf algebras, but in fact an appropriate
version of Theorem~\ref{ith:lg-Gproj-ha} holds for any arbitrary
finite projective $R$-algebras $A$; see Theorem~\ref{th:lg-kproj}. In
such contexts the natural cohomology ring to consider, vis-\`a-vis
stratification, is the Hochschild cohomology of $A$ over $R$. It is
plausible that the analogue of the map $\kappa^a_\fp$ in that context
is the key to patching fibrewise stratification, when available, to
get a global stratification result for $A$.

\section{A fibrewise criterion for localising subcategories}
\label{se:local-global}

Throughout $R$ is a commutative noetherian ring and $A$ a \emph{finite projective $R$-algebra}; that is to say $A$ is an $R$-algebra that is finitely generated and projective as an $R$-module.  In particular, as a ring $A$ is noetherian on the left  and on the right. Given an $A$-complex $X$ and a point $\fp$ in $\Spec R$, the Zariski spectrum of $R$, set
\[
\fibre X{\fp}\colonequals X\otimes_R k(\fp)
\]
viewed as a complex of $\fibre A{\fp}$-modules. The assignment $X\mapsto \fibre X{\fp}$ is exact on the homotopy category of $A$-complexes, and hence also on any of its subcategories, in particular, on $\KProj A$, the homotopy category of projective $A$-modules. The latter has a natural structure of a triangulated category, with arbitrary coproducts. Given an $A$-complex $Y$ in $\KProj A$ we write $\Loc(Y)$ for the localising subcategory of the homotopy category generated by $Y$.  The main result in this section is the following \emph{fibrewise criterion} for detecting objects in $\Loc(Y)$. 

\begin{theorem}
\label{th:lg-kproj}
Suppose that $R$ is regular. Let $A$ be a finite projective $R$-algebra, and let $X,Y$ be objects in $\KProj A$. The following conditions are equivalent. 
\begin{enumerate}[\quad\rm(1)]
\item
 $X\in  \Loc(Y)$ in $\KProj A$;
\item
 $\fibre X\fp\in \Loc(\fibre Y{\fp})$ in $\KProj{\fibre A\fp}$ for each $\fp\in \Spec R$.
\end{enumerate}
\end{theorem}

The ring $R$ is by definition \emph{regular} if every finitely
generated $R$-module has finite projective dimension.  An equivalent
condition is that the ring $R_\fp$ has finite global dimension for
each $\fp$ in $\Spec R$. The global dimension of $R$ is then equal to
$\dim R$, its Krull dimension.

The statement above is inspired by,
and extends, an analogous statement for the derived category of $A$,
established in \cite{Gnedin/Iyengar/Krause:2022a}. It yields also a
statement about the stable category of Gorenstein projective
$A$-modules; see Theorem~\ref{th:lg-GProj}.

The proof of the theorem above takes some preparation and is given towards the end of this section. We start by  recalling some properties of the homotopy category of  projective modules.

\subsection*{The homotopy category of projectives}
For the moment, $A$ can be any ring that is noetherian on both sides;
that is to say, $A$ is noetherian as a left and as a right $A$-module.
For us, $A$-modules mean left $A$-modules, and $\op A$-modules are
identified with right $A$-modules. When $A$ is an $R$-algebra for some
commutative ring $R$, then it is convenient
to write for any $A$-module the $R$-action on the opposite side.  We
denote by $\Mod A$ the (abelian) category of $A$-modules and by
$\rmod A$ its full subcategory consisting of finitely generated
modules. The full subcategory of $\Mod A$ consisting of projective
modules is denoted $\Prj A$.

For any additive category $\cat A\subseteq\Mod A$, like the ones in the last paragraph, $\bfK(\cat A)$ will denote the associated homotopy category, with its natural structure as a triangulated category. Morphisms in this category are denoted $\Hom_{\bfK(A)}(-,-)$.  An object $X$ in $\bfK(\cat A)$ is \emph{acyclic} if $\hh X=0$, and the full subcategory of acyclic objects in $\bfK(\cat A)$ is denoted $\bfK_{\mathrm{ac}}(\cat A)$.

Let $\dcat A$ denote the derived category of $A$-modules and $\dbcat A$ the bounded derived category of $\rmod A$. Let $\bfq\colon \KMod A\to \dcat A$ be the localisation functor; its kernel is $\bfK_{\mathrm{ac}}(\Mod A)$.  We write $\bfq$ also for its restriction to the homotopy category of projective modules. This functor has an adjoint:
\begin{equation}\label{eq:p-q}
\begin{tikzcd}
\KProj A    
  \arrow[twoheadrightarrow,yshift=-.75ex]{r}[swap]{\bfq}
  &   \arrow[tail,yshift=.75ex]{l}[swap]{\pres} \dcat A\,.
\end{tikzcd}
\end{equation}
Our convention is to write the left adjoint above the corresponding right one. In what follows it is convenient to conflate $\pres$ with  $\pres\circ\bfq$.
 The image of $\pres$, denoted $\kproj A$, consists precisely of the K-projective complexes, namely, those complexes $P$ such that $\Hom_{\bfK(A)}(P,-)=0$ on acyclic complexes in $\bfK(\Mod A)$.  

\subsection*{Compact objects}
The category $\KProj A$ is triangulated, admits arbitrary direct sums, and is compactly generated.  As in any triangulated category with arbitrary direct sums, an object $X$ in $\KProj A$ is \emph{compact} if $\Hom_{\bfK(A)}(X,-)$ commutes with direct sums. The compact objects in $\KProj A$ form a thick subcategory,
denoted $\KProjc A$. The assignment $X\mapsto \Hom_{\op A}(\pres M,A)$ induces an equivalence
\begin{equation}
\label{eq:kprojc}
\iota_A\colon \op{\dbcat{\op A}} \longiso \KProjc A.
\end{equation}
This result is  due to
J{\o}rgensen~\cite[Theorem~3.2]{Jorgensen:2005a}; see also
\cite{Iyengar/Krause:2006a}.

\subsection*{Regular rings}

Recall that the ring $R$ is regular if every finitely generated
$R$-module has finite projective dimension, that is, every complex in
$\dbcat R$ is perfect.  We record a couple of basic facts for later
use.

\begin{lemma}
\label{le:regular}
Let $R$ be a commutative noetherian ring.
\begin{enumerate}[\quad\rm(1)]
\item  The ring $R$ is regular if and only if every
  acyclic complex in $\KProj R$ is null-homotopic.
  \item Suppose that $R$ is regular and local of Krull dimension $d$ with residue field
    $k$. Then $\RHom_R(k,R)\iso \Sigma^d k$.
    \end{enumerate}
\end{lemma}

\begin{proof}
(1) The functor $\pres$ in \eqref{eq:p-q} restricted to compact objects embeds the perfect complexes over $R$ into $\dbcat R$, via \eqref{eq:kprojc}. This embedding is an equivalence if and only $\pres$ is an equivalence. Clearly,  $\pres$ is an
equivalence if and only if $\bfq$ is an equivalence. See also \cite{Iacob/Iyengar:2009a}.

(2) The Koszul complex $K$ on a minimal generating set for the maximal
ideal of $R$ provides a projective resolution $K\to k$. Since $\Hom_R(K,R)\cong \Sigma^{-d}K$, by
\cite[Proposition~1.6.10]{Bruns/Herzog:1998a}, the stated assertion follows. 
\end{proof}

\subsection*{Fibres and the Koszul complex}

We return to the context of finite projective algebras over a
commutative ring $R$ and wish to describe the functor
$-\otimes_R k(\fp)$ for each prime $\fp$ in $\Spec R$ via the Koszul
complex for $\fp$. We reduce to the local case.

Let $(R,\fm,k)$ be a local ring, with maximal ideal $\fm$ and residue
field $k$. In this paragraph, we write $A_k$ instead of
$A_{k(\fm)}$. Let $\pi$ be the functor from $\KProj A$ to
$\KProj{A_k}$ given by the assignment
$X\mapsto X_k\colonequals X\otimes_Rk$. It is clear that $\pi$
preserves coproducts and so has a right adjoint by Brown
representability, say $\pi_r$. Thus there is
the adjoint pair
\begin{equation}
\label{eq:adjunction}
\begin{tikzcd}
    \KProj A \arrow[yshift=.5ex]{r}{\pi} 
    	& \KProj{A_k}\,. \arrow[yshift=-.5ex]{l}{\pi_r}
\end{tikzcd}
\end{equation}

The following observation is useful in the sequel.

\begin{lemma}
\label{le:compacts}
Consider an adjoint pair of functors 
$
\begin{tikzcd}
  \cat T \arrow[yshift=.5ex]{r}{\pi} 
    	& \cat U \arrow[yshift=-.5ex]{l}{\pi_r}
\end{tikzcd}
$ between compactly generated triangulated categories. Then $\pi$
preserves all coproducts if and only if $\pi$
preserves compact objects. In that case the restriction
$ \pi^c\colon{\cat T}^{c}\to {\cat U}^{c}$  admits a right adjoint if
and only if $\pi_r$ preserves compact objects.
\end{lemma}

\begin{proof}
  The first assertion is well-known; see \cite[Theorem~5.1]{Neeman:1996a}. Now assume  that $\pi$
  preserves compact objects and that $ \pi^c\colon{\cat T}^{c}\to
{\cat U}^{c}$ admits a right adjoint, say $\rho\colon{\cat U}^{c}\to {\cat
  T}^{c}$. Then for any  $X\in\cat T^c$ and $Y\in\cat U^c$ the natural
bijection
\[\Hom_{\cat T}(X,\rho(Y))\cong \Hom_{\cat U}(\pi(X),Y)\cong
  \Hom_{\cat T}(X,\pi_r(Y))
  \] 
  maps the identity  $\id_{\rho(Y)}$ to a morphism $\phi\colon\rho(Y)\to\pi_r(Y)$. The map $\Hom_{\cat T}(X,\phi)$ is a
bijection for all $X$. Since $\cat T$ is compactly generated,  $\phi$ is an isomorphism.

The other implication is clear.
\end{proof}

We need to understand the unit $\id \to \pi_r\pi$ and counit
$\pi\pi_r\to \id$ of the adjunction~\eqref{eq:adjunction}. To that
end, consider the Koszul complex, $K$, on a minimal generating set for
the ideal $\fm$. The map $R\to K$ of $R$-complexes induces a natural
map
\begin{equation}
\label{eq:unit}
X\cong X\otimes_R R \lra X\otimes_RK\qquad\text{for $X$ in $\KProj A$.}
\end{equation}
On the other hand, the augmentation $R\to k$ factors through $R\to K$
via a map $K\to k$, which induces
$V\colonequals K\otimes_Rk\to k\otimes_R k\iso k$ and therefore a
natural map
\begin{equation}
\label{eq:counit}
V\otimes_k Y \lra  k\otimes_k Y \cong Y \qquad\text{for $Y$ in $\KProj{A_k}$.}
\end{equation}
The result below is the key to the proof of Theorem~\ref{th:lg-kproj} and also in other computations that follow.

\begin{lemma}
\label{le:unit}
When $R$ is regular, the following statements hold.
\begin{enumerate}[\quad\rm(1)]
\item
Both $\pi$ and $\pi_r$ preserve compact objects and arbitrary direct sums.
 \item
 Restricting the adjunction \eqref{eq:adjunction} to compact objects gives the adjunction
\[
\begin{tikzcd}[column sep=large]
    \dbcat{\op A} \arrow[yshift=-.7ex]{r}[swap]{-\lotimes_R k} 
    	& \dbcat{(\op A_k)} \arrow[yshift=.7ex]{l}[swap]{\Sigma^{-d}\rho}
\end{tikzcd}
\]
where $\rho$ is induced by restriction along the map $A\to A_k$ and $d\colonequals \dim R$.
\item
 The maps \eqref{eq:unit} and \eqref{eq:counit} are the unit and the counit, respectively, of the adjunction~\eqref{eq:adjunction}. 
\item
For $X\in\KProj A$ and $Y\in \KProj{A_k}$, there are natural isomorphisms
\[
\pi_r\pi(X) \cong X\otimes_RK \qquad\text{and}\qquad \pi\pi_r(Y)\cong V\otimes_k Y\,.
\]
In particular, $\pi\pi_r(Y)$ is isomorphic to a finite direct sum of shifts of $Y$.
\end{enumerate}
\end{lemma}

\begin{proof}
(1) and (2) We have already observed that $\pi$ preserves arbitrary direct sums. We verify that $\pi$ preserves compacts, equivalently that  its right adjoint, $\pi_r$, preserves arbitrary direct sums.   Any compact object in $\KProj A$ is of the form $\iota_AM =\Hom_A(\pres_{\op A}M,A)$ for some $M\in \dbcat{\op A}$. A direct computation gives 
\begin{align*}
\pi \Hom_A(\pres_{\op A}M,A) 
	&= \Hom_A(\pres_{\op A}M,A) \otimes_R k \\
	&\cong \Hom_{A_k}((\pres_{\op A}M)\otimes_Rk, A_k) \\
	&\cong \Hom_{A_k}(\pres_{\op A}(k\lotimes_RM), A_k)\,.
\end{align*}
Since $R$ is regular, $k$ is in $\Thick(R)$ in $\dcat R$ and hence $k\lotimes_RM$ is in $\Thick(M)$ in $\dbcat{\op A}$.  Therefore the $A_k$-module $H^*(k\lotimes_RM)$ is finitely generated, that is to say, $k\lotimes_RM$ is in $\dbcat{\op A_k}$. Thus the complex 
\[
\Hom_{A_k}(\pres_{\op A}(k\lotimes_RM), A_k) = \iota_{A_k}(k\lotimes_RM)
\]
in $\KProj{A_k}$ is compact. This justifies the claim that $\pi$ preserves compacts. Along the way we have established that
\[
\pi \iota_A(M) \cong \iota_{A_k}(k\lotimes_RM) 
\]
for $M$ in $\dbcat{\op A}$.  In other words, restricted to compacts $\pi$ is the functor 
\[
k\lotimes_R-\colon \dbcat{\op A} \lra \dbcat{\op A_k}
\]
via the identification in \eqref{eq:kprojc}. The functor $\Sigma^{-d}\rho$ is left adjoint to the functor above:
\begin{align*}
\Hom_{\mathbf{D}(A)}(\Sigma^{-d}\rho N, M) 
	&\cong \Hom_{\mathbf{D}(A_k)}(\Sigma^{-d}N, \RHom_R(k,M))\\
	&\cong \Hom_{\mathbf{D}(A_k)}(\Sigma^{-d}N, \RHom_R(k,R)\lotimes_RM )\\
	&\cong \Hom_{\mathbf{D}(A_k)}(\Sigma^{-d}N, \Sigma^{-d} k\lotimes_RM) \\
	&\cong \Hom_{\mathbf{D}(A_k)}(N,  k\lotimes_RM) 
\end{align*}
The first isomorphism is standard adjunction, the second uses the fact
that $k$ is perfect as an $R$-complex, since $R$ is regular, and the
third one follows from the fact that $\RHom_R(k,R)\iso \Sigma^dk$; see
Lemma~\ref{le:regular}.

At this point we can apply Lemma~\ref{le:compacts} to deduce that $\pi_r$ also  preserves compacts and is isomorphic to the functor $\Sigma^{-d}\rho$ when restricted to compact objects. This settles both (1) and (2).

(3) The canonical map $K\otimes_Rk\to k\otimes_Rk\iso k$ induces the map 
\[
X\otimes_RK\otimes_Rk\lra X\otimes_Rk\,,
\]
and this corresponds under the adjunction isomorphism
\[
\Hom_{\bfK(A_k)}(X\otimes_RK\otimes_Rk, X\otimes_Rk)\cong \Hom_{\bfK(A)} (X\otimes_RK, \pi_r(X\otimes_Rk)) 
\]
to a natural map
\begin{equation}
\label{eq:AB-kproj}
X\otimes_RK \lra \pi_r\pi(X)\,.
\end{equation}
It suffices to prove that this is an isomorphism.  We verify this when
$X$ is compact; the general case then follows as $\pi$ and $\pi_r$
preserve arbitrary direct sums, by (1). This brings us to the
adjunction in (2). Then \eqref{eq:AB-kproj} is the map
\[
X\otimes_RK \lra X\lotimes_Rk 
\]
which is an isomorphism because the map $K\to k$ is an isomorphism in $\dcat R$; here again we are using the hypothesis that $R$ is regular. This completes the proof that \eqref{eq:unit} is the unit of the adjunction.

The claim about \eqref{eq:counit} can be verified along the same lines. 

(4) This is a direct consequence of (3). The last assertion follows
from the fact that $V$ is a finite  graded $k$-vector space with zero differential.
\end{proof}

\subsection*{Local cohomology and support}
In the proof of Theorem~\ref{th:lg-kproj}, and later on in the sequel,
we require the theory of local cohomology and support from
\cite{Benson/Iyengar/Krause:2008a}, with respect to the action of the
ring $R$ on the homotopy category of projective modules. The analogue
for the homotopy category of injective modules is described in
\cite[\S7]{Iyengar/Krause:2022a}. One could invoke that theory, for
the two homotopy categories are equivalent, at least under certain
minor additional constraints on $A$, but to keep this manuscript
self-contained we develop the needed results for $\KProj A$ directly.
 
For any pair of objects $X,Y$ in $\KProj A$ there is a natural $R$-module structure on
$\Hom_{\bfK(A)}(X,Y)$, so that $\KProj A$ is an $R$-linear triangulated
category, in the sense of \cite[\S4]{Benson/Iyengar/Krause:2008a}. In
particular, for each specialisation closed subset $V$ of $\Spec R$
there is an exact triangle
\begin{equation}
\label{eq:localisation-triangle}
\gam_V X \lra X\lra L_V X\lra 
\end{equation}
such that the object $\gam_{V}X$ is $V$-torsion and
$\Hom_{\bfK(A)}(-,L_VX)=0$ on $V$-torsion objects. Here an object $Y$ is by
definition \emph{$V$-torsion} 
if for each compact object $C$ in $\KProj A$ the $R$-module
$\Hom_{\bfK(A)}(C,Y)$ is $V$-torsion.

For any ideal  $I\subset R$ we consider the closed set
\[V(I) \colonequals\{\fp\in\Spec R\mid I\subseteq \fp\}\,.\]

\begin{lemma}
\label{le:gamma-I}
Let $I\subset R$ be an ideal and $K$ the Koszul complex on a finite generating set for the ideal $I$. For each $X$ in $\KProj A$ one has 
\[
\Loc(\gam_{V(I)}X)=\Loc(X\otimes_RK)\,.
\]
\end{lemma}

\begin{proof}
If $K$ is the Koszul complex on a single element $r$ in $R$, the
complex $X\otimes_RK$ is isomorphic to the mapping cone of the
morphism $X\xra{r}X$; in other words, $X\otimes_RK$ is the complex
denoted $\kos Xr$ in \cite[\S2.5]{Benson/Iyengar/Krause:2011a}.
The Koszul complex $K$ on a sequence of elements $r_1,\dots,r_n$ generating the ideal $I$ can be constructed as an iterated mapping cone, so $X\otimes_RK$ represents $\kos XI$. Thus the stated result is a special case of \cite[Proposition~2.11(2)]{Benson/Iyengar/Krause:2011a}.
\end{proof}

Fix a point $\fp$ in $\Spec R$ and consider the specialisation closed subset 
\[
Z(\fp)\colonequals \Spec R \setminus \Spec R_\fp = \{\fq\in\Spec R\mid \fq\not\subseteq \fp\}\,.
\]
The localisation functor $X\to L_{Z(\fp)}X$ models localisation at $\fp$ in the sense that for each compact object $C$ in $\KProj A$ the 
 map
\[
\Hom_{\bfK(A)}(C,X)\lra \Hom_{\bfK(A)}(C,L_{Z(\fp)}X)
\]
of $R$-modules induces an isomorphism of $R_\fp$-modules
\[
\Hom_{\bfK(A)}(C,X)_\fp\longiso \Hom_{\bfK(A)}(C,L_{Z(\fp)}X)\,.
\]
See \cite[Theorem~4.7]{Benson/Iyengar/Krause:2008a}, and also \cite[Proposition~2.3]{Benson/Iyengar/Krause:2011a}. 

The localisation functor $L_{Z(\fp)}$ admits an alternative
description which will be useful. For a complex $X$ in $\KProj A$ set
\[X_\fp\colonequals X\otimes_R R_\fp\] viewed as a complex of $A_\fp$-modules,
where $A_\fp$ denotes the $R_\fp$-algebra $A\otimes_R R_\fp$.
The assignment $X\mapsto X_\fp$ yields an adjoint pair of functors
\begin{equation}
\label{eq:localisation}
\begin{tikzcd}
    \KProj A \arrow[yshift=.5ex]{r}{\lambda} 
    	& \KProj{A_\fp}\,. \arrow[yshift=-.5ex]{l}{\lambda_r}
\end{tikzcd}
\end{equation}
The right adjoint $\lambda_r$ exists as localisation preserves coproducts.

\begin{lemma}
\label{le:loc-description}
The right adjoint $\lambda_r$ preserves coproducts. Moreover, for each
$X$ in $\KProj A$, the unit $X\to \lambda_r\lambda(X)$ of the  adjunction \eqref{eq:localisation} is naturally isomorphic to the localisation $X\to L_{Z(\fp)}(X)$, so that
\[
L_{Z(\fp)}(X)\cong \lambda_r\lambda(X)\,.
\]
\end{lemma}

\begin{proof}
Observe that the functor $\lambda$ also preserves compact objects, hence its right adjoint $\lambda_r$ preserves coproducts.   
As to the second claim, it suffices to verify that for any compact object $C$ in $\KProj A$ the map
\[
\Hom_{\bfK(A)}(C,X) \lra \Hom_{\bfK(A)}(C,\lambda_r\lambda(X)) 
\]
induced by the unit is localisation at $\fp$. Adjunction gives an isomorphism
\[
\Hom_{\bfK(A)}(C,\lambda_r\lambda(X)) \cong \Hom_{\bfK(A_\fp)}(\lambda C,\lambda(X)) = \Hom_{\bfK(A_\fp)} (C_\fp,X_\fp)
\]
 of $R$-modules, so the module on the left is $\fp$-local. Thus one gets an induced map
\[
\Hom_{\bfK(A)}(C,X)_\fp \lra \Hom_{\bfK(A)}(C,\lambda_r\lambda(X)) 
\]
and the desired result is that this is an isomorphism. Since $C$ is compact, and localisation at $\fp$, and the functors $\lambda_r$ and $\lambda$ preserve coproducts,  it suffices to verify the map above is an isomorphism when $X$ is also compact. Consider again  the adjunction isomorphism
\[
\Hom_{\bfK(A)}(C,\lambda_r\lambda X)  \cong \Hom_{\bfK(A_\fp)} (C_\fp, X_\fp)\,.
\]
Since $C_\fp$ and $X_\fp$ are compact in $\KProj{A_\fp}$, by the description of compact objects in $\KProj A$, the desired result is that for $M,N$ in $\dbcat{\op A}$, the map
\[
\Hom_{\mathbf D(\op A)}(M,N)_\fp \lra \Hom_{\mathbf D(\op A_\fp)}(M_\fp,N_\fp)
\]
is an isomorphism. But this is clear. 
\end{proof}

The  \emph{local cohomology functor} at $\fp$ is the functor
$\gam_\fp$ on $\KProj A$ given by
\begin{equation}
\label{eq:gammap}
\gam_{\fp}(X) \colonequals \gam_{V(\fp)}L_{Z(\fp)}(X) \cong \lambda_{r} \gam_{V(\fp R_{\fp})}(X_\fp)\,,
\end{equation}
where the isomorphism is the one from Lemma~\ref{le:loc-description}.

The local cohomology functors reduce the description of
localising subcategories to a local problem,
because the \emph{local-to-global theorem} 
says that
\begin{equation}\label{eq:local-to-global}
  \Loc(X)=\Loc(\{\gam_{\fp}X\mid \fp\in\Spec R\})\quad\text{for
  }X\in\KProj A;
\end{equation}
see \cite[\S3]{Benson/Iyengar/Krause:2011a} and also
\cite[Theorem~6.9]{Stevenson:2013a}.

\begin{proof}[Proof of Theorem~\ref{th:lg-kproj}]
The implication (1)$\Rightarrow$(2) is clear since for each $\fp$ in
$\Spec R$ the functor given by  $X\mapsto \fibre X\fp$ is exact and preserves all coproducts.

(2)$\Rightarrow$(1) Let $X,Y$ be objects in $\KProj A$.  By the
local-to-global theorem~\eqref{eq:local-to-global} it suffices to verify for each 
$\fp\in\Spec R$ that  $\fibre X\fp\in\Loc(\fibre Y{\fp})$ in
$\KProj{\fibre A\fp}$ implies $\gam_{\fp} X\in\Loc(\gam_{\fp}Y)$ in
$\KProj A$.

We denote by $K$ the Koszul complex on a minimal generating set for
the maximal ideal $\fp R_\fp$ of $R_\fp$. We have
\[\fibre X\fp=X_\fp\otimes_{R_\fp}k(\fp R_\fp)\]
and therefore $\fibre X\fp\in\Loc(\fibre Y{\fp})$ implies
\[X_\fp\otimes_{R_\fp} K\in\Loc(Y_\fp\otimes_{R_\fp} K)\quad
\text{in}\quad\KProj {A_\fp}\] by Lemma~\ref{le:unit}. This means
\[\gam_{V(\fp R_\fp)}(X_\fp)\in\Loc(\gam_{V(\fp R_\fp)}(Y_\fp))\quad\text{in}\quad\KProj{A_\fp}\] by Lemma~\ref{le:gamma-I}. It remains to apply
the functor $\lambda_r$. Thus
\[\lambda_r\gam_{V(\fp R_\fp)}(X_\fp)\in\Loc(\lambda_r\gam_{V(\fp R_\fp)}(Y_\fp))\quad
  \text{in}\quad\KProj A\,.\qedhere
 \]
\end{proof}

\begin{chunk}
\label{ch:lg-finite-version}
The derived category $\dcat A$ identifies with a localising subcategory of $\KProj A$ via the left adjoint of the canonical functor  $\KProj A\to\dcat A$. Thus Theorem~\ref{th:lg-kproj} implies the analogous  description of localising subcategories of $\dcat A$ from
\cite{Gnedin/Iyengar/Krause:2022a}. Here is another noteworthy consequence.

\begin{corollary}
\label{co:lg-finite-version}
Let $R$ be a regular ring, $A$ a finite projective $R$-algebra, and $M,N$ in $\dbcat A$. The following conditions are equivalent.
\begin{enumerate}[\quad\rm(1)]
\item
 $M\in  \Thick(N)$ in $\dbcat A$;
\item
 $\fibre M\fp\in \Thick(\fibre N{\fp})$ in $\dbcat{\fibre A\fp}$ for each $\fp\in \Spec R$.
\end{enumerate}
\end{corollary}

\begin{proof}
For compact objects $X,Y$ in any compactly generated triangulated category, one has $X\in\Thick(Y)$ if and only if $X\in \Loc(Y)$; see, for instance, \cite[Lemma~2.2]{Neeman:1992b}. Thus the desired result is an immediate consequence of Theorem~\ref{th:lg-kproj} and equivalence \ref{eq:kprojc}, applied to $\op A$.
\end{proof}
\end{chunk}

The preceding result applied with $N=A$ implies that $M$ is perfect if
and only if it is fibrewise perfect. Here is a more precise result, for later use.

\begin{lemma}
\label{le:projective-fibrewise}
Let $R$ be a commutative noetherian ring, $A$ a finite projective $R$-algebra, and $M$ a finitely generated $A$-module. When $M$ is projective as an $R$-module, there is an equality
\[
\projdim_AM = \sup\{\projdim_{\fibre A{\fp}}{\fibre M{\fp}}\mid \fp\in \Spec R\}\,.
\]
Moreover, it suffices to take the supremum over the maximal ideals in $R$.
\end{lemma}

\begin{proof}
Even without the hypothesis that $M$ is finite projective over $R$, one has
\[
\projdim_AM = \sup\{\projdim_{A_\fp}{ M_\fp}\mid \fp\in \Spec R\}
\]
by \cite[Corollary~III.6.6]{Bass:1968a}. Thus replacing $R$, $A$, and $M$ by their localisations at $\fp$ we can assume $R$ is a local ring, say with maximal ideal $\fm$ and residue field $k$. Then the desired result is that
\[
\projdim_AM = \projdim_{A_k}{M_k}\,.
\]
Since $A$ is semi-local, \cite[Proposition~A.1.5]{Avramov/Iyengar:2021a} yields
\[
\projdim_AM =\max_{1\le j\le r}\{i\in\mathbb{N} \mid \Ext^i_A(M,L_j)\ne 0\}
\]
where $L_1,\dots,L_r$ are the simple $A$-modules. The $L_j$ are modules over $A_k\colonequals A/\fm A$ and $M$ is projective over $R$, so adjunction yields isomorphisms
\[
\Ext^i_A(M,L_j) \cong \Ext^i_{A_k}(M_k,L_j)\,.
\]
Since the $L_j$ are the simple modules over $A_k$ the desired result follows.
\end{proof}

\section{Gorenstein algebras}
\label{se:Gorenstein}
Let $R$ be a commutative noetherian ring. Following~\cite{Gnedin/Iyengar/Krause:2022a, Iyengar/Krause:2022a}, we say $A$ is a \emph{Gorenstein} $R$-algebra when it is a finite projective $R$-algebra such that for each $\fp$ in $\supp_RA$ the ring $A_\fp$ has finite injective dimension on the left and on the right; that is to say, it is Iwanaga--Gorenstein.  When this holds $R_\fp$ is Gorenstein for each $\fp$ in $\supp_RA$.  Here is a characterisation of the Gorenstein property that is in the spirit of this work; see also \cite[Theorem~6.8]{Gnedin/Iyengar/Krause:2022a}.

\begin{proposition}
\label{pr:gor=fibre}
Let $R$ be a commutative Gorenstein ring and $A$ a finite projective $R$-algebra. Then $A$ is Gorenstein if and only the finite dimensional algebra $\fibre A{\fp}$ is Iwanaga--Gorenstein for each $\fp$ in $\Spec R$.
\end{proposition}

\begin{proof}
By \cite[Theorem~4.6]{Iyengar/Krause:2022a}, the $R$-algebra $A$ is Gorenstein if and only if the $A$-bimodule $\Hom_R(A,R)$ is perfect on both sides. Since the $A$-module $\Hom_R(A,R)$ is finitely generated on both sides, it is perfect if and only if the $\fibre A{\fp}$-module
\[
\fibre {\Hom_R(A,R)}{\fp} \cong \Hom_{k(\fp)}(\fibre A{\fp},k(\fp))
\]
is perfect on both sides, for each $\fp$ in $\Spec R$; this follows from Corollary~\ref{co:lg-finite-version} applied with $M\colonequals \Hom_R(A,R)$ and $N\colonequals A$. It remains to observe that this latter condition is equivalent to $\fibre A{\fp}$ being Iwanaga--Gorenstein.
\end{proof}

\begin{chunk}
One consequence of the Gorenstein condition is that complexes in $\KacProj A$ are \emph{totally acyclic}, namely each complex $X\in \KacProj A$ satisfies
\[
\Hom_{\bfK(A)}(X,P)=0 \quad\text{for any projective $A$-module $P$.}
\]
See \cite[Theorem~5.6]{Iyengar/Krause:2022a} for a proof.  The functors 
\[
\KacProj A\xra{\mathrm{incl}} \KProj A\xra{\ \bfq\ }\dcat A
\]
induce a recollement of triangulated categories
\begin{equation}
\label{eq:recollement}
\begin{tikzcd}
\KacProj A   \arrow[hookrightarrow]{rr}[description]{\mathrm{incl}} && \KProj A
  \arrow[twoheadrightarrow,yshift=-1.5ex]{ll}
  \arrow[twoheadrightarrow,yshift=1.5ex]{ll}[swap]{\bft}
  \arrow[twoheadrightarrow]{rr}[description]{\bfq} &&\dcat A
  \arrow[tail,yshift=-1.5ex]{ll}
  \arrow[tail,yshift=1.5ex]{ll}[swap]{\pres}
\end{tikzcd}
\end{equation}
The functor $\bft$, left adjoint to the inclusion of the acyclic
complexes of projectives, associates to each complex its
\emph{complete resolution}.

From \cite[Theorem~4.6]{Iyengar/Krause:2022a} one gets an equivalence:
\[
\RHom_A(-,A)\colon {\dbcat A}^{\mathrm{op}} \longiso \dbcat {\op A}\,.
\] 
Composing this with the equivalence  \eqref{eq:kprojc} yields a canonical equivalence
\begin{equation}
\label{eq:Gorenstein-kprojc}
\dbcat A \longiso \KProjc A\,.
\end{equation}
\end{chunk}

\subsection*{Gorenstein projective modules}
An $A$-module $M$ is \emph{Gorenstein projective} if it occurs as a
syzygy module in an acyclic complex of projective $A$-modules. Thus,
there is some $X\in \KacProj A$ such that $M\cong\Coker(d^{-1}_X)$. We
write $\GProj A$ for the full subcategory of $\Mod A$ consisting of
the Gorenstein projective modules, and $\Gproj A$ for its subcategory
of finitely generated modules. Both these are Frobenius categories,
with projective and injective objects the projective modules in the
corresponding categories; see for example
\cite[Proposition~7.2]{Krause:2005a}. The corresponding stable
categories are denoted $\uGProj A$ and $\uGproj A$, respectively. The
first part of the result below was proved by
Buchweitz~\cite[Theorem~4.4.1]{Buchweitz:2021a} when $A$ is
Iwanaga--Gorenstein, but the same argument carries over to this
context.

\begin{theorem}
\label{th:Buchweitz-big}
The assignment $X\mapsto \Coker(d^{-1}_X)$ induces an equivalence of
$R$-linear triangulated categories $\KacProj A\iso \uGProj
A$. Moreover, these categories are compactly generated, and $\uGproj A$
identifies with the full subcategory of compact objects of
$\uGProj A$.
\end{theorem}
\begin{proof}
  In the dual setting of Gorenstein injectives 
  the first assertion is \cite[Proposition~7.2]{Krause:2005a}. In
  fact, we have an equivalence $\KInj A\iso\KProj A$ by
  \cite[Theorem~5.6]{Iyengar/Krause:2022a} and then the second
  assertion follows from \cite[Theorem~6.5]{Iyengar/Krause:2022a}.
\end{proof}

The Gorenstein projectivity of a module is inherited by its
fibres. Without further restrictions, the converse need not hold.

\begin{lemma}
\label{le:gproj-fibre}
Let $R$ be a regular ring and $A$ a Gorenstein $R$-algebra.  If an $A$-module $M$ is Gorenstein projective, then so is the $\fibre A{\fp}$-module  $\fibre M{\fp}$ for each $\fp$ in $\supp_RA$.
\end{lemma}

\begin{proof}
Let $X$ be an acyclic complex of projective $A$-modules in which $M$
is a syzygy. Since $R$ is regular and  $X$ is in $\KacProj R$, it is
split-acyclic as an $R$-complex by Lemma~\ref{le:regular}. Thus $\fibre X{\fp} = X\otimes_R k(\fp) $ is also split-acyclic, and in particular acyclic. It consists of projective $\fibre A{\fp}$-modules and $\fibre M{\fp}$ is a syzygy module in it, so the latter is Gorenstein projective.
\end{proof}

In view of Theorem~\ref{th:Buchweitz-big} and the preceding
result, one gets an analogue of  Theorem~\ref{th:lg-kproj} for
Gorenstein projective modules.

\begin{theorem}
\label{th:lg-GProj}\pushQED{\qed}
Let $R$ be a regular ring and $A$  a Gorenstein $R$-algebra. For Gorenstein
projective $R$-modules $X,Y$ we have in $\uGProj A$
\[X\in \Loc(Y)\quad\iff\quad\fibre X\fp\in
  \Loc(\fibre Y{\fp}) \text{ for each }\fp\in \Spec R.\qedhere\]
\end{theorem}

In the remainder of this section, we focus on a class of Gorenstein
algebras for which it is easy to describe the Gorenstein projective
modules.

\subsection*{Fibrewise self-injective algebras}
The \emph{dualising bimodule} of a finite projective $R$-algebra $A$ is the $A$-bimodule 
\[
\omega_{A/R}\colonequals \Hom_R(A,R)\,.
\]
As noted in the proof of Proposition~\ref{pr:gor=fibre}, when $A$ is a
Gorenstein $R$-algebra, $\omega_{A/R}$ is perfect on either side,
though not necessarily as a bimodule.  Moreover, this property
characterises the Gorenstein property of $A$ when $R$ is Gorenstein;
see \cite[Theorem~4.6]{Iyengar/Krause:2022a}. In the sequel,
Gorenstein algebras for which the dualising bimodule is projective on
either side play a prominent role. The result below characterises
these algebras in terms of their fibres.

An $R$-algebra $A$ is \emph{fibrewise self-injective} if it is a finite
projective $R$-algebra such that the finite dimensional algebra
$\fibre A{\fp}$ is self-injective for each $\fp$ in $\supp_R A$. Our
primary example is the group algebra of a finite group scheme over
$R$.

For a module $M$ we denote by $\add(M)$ the full subcategory of finite
direct sums of copies of $M$ plus their direct summands.

\begin{lemma}
\label{le:fibrewise-injective}
Let $R$ be a Gorenstein ring and $A$ a finite projective $R$-algebra. The  conditions below are equivalent.
\begin{enumerate}[\quad\rm(1)]
\item
The $R$-algebra $A$ is fibrewise self-injective;
\item
$\mathrm{add}(\omega_{A/R})=\add(A)$ and $\mathrm{add}(\omega_{\op{A}/R})=\add(\op A)$;
\item 
The dualising bimodule $\omega_{A/R}$ is projective on the left and on the right.
\end{enumerate}
If they hold the $R$-algebra $A$ is Gorenstein, and one has an equivalences of categories
\[
\omega_{A/R}\otimes_A (-) \colon \Prj A \longiso \Prj A 
\]
with inverse $\Hom_A(\omega_{A/R},-)$, and similarly for $\Prj\op A$.
\end{lemma}

\begin{proof}
  In what follows we use the observation that for each $\fp$ in
  $\Spec R$ one has an isomorphism of $\fibre A{\fp}$-bimodules:
\[
\omega_{A/R}\otimes_R k(\fp) \cong \omega_{{\fibre A{\fp}}/k(\fp)}\,.
\]
Hence the $A$-bimodule $\omega_{A/R}$ is projective on either side if
and only if the $\fibre A{\fp}$-bimodule
$\omega_{{\fibre A{\fp}}/k(\fp)}$ is projective on either side for
each $\fp$ in $\Spec R$; see Lemma~\ref{le:projective-fibrewise}.
 
(1)$\Rightarrow$(2) It suffices to prove that $A$ is in
$\add(\omega_{A/R})$, equivalently that the map
\[
\omega_{A/R}\otimes_A\Hom_A(\omega_{A/R},A) \lra A
\]
given by evaluation, is surjective. Since an $R$-module $M$ is zero if and only if $\fibre M\fp$ is zero for each $\fp$ in $\Spec R$, it suffices to check the surjectivity of the map on the fibres, that is to say, the map 
\[
\omega_{\fibre A{\fp}/k(\fp)}\otimes_{\fibre A{\fp}} \Hom_{\fibre A{\fp}}(\omega_{\fibre A{\fp}/k(\fp)},{\fibre A{\fp}}) \lra {\fibre A{\fp}}
\]
is surjective. This holds as the  $k(\fp)$-algebra ${\fibre A{\fp}}$ is self-injective.

(2)$\Rightarrow$(3) is clear.

(3)$\Rightarrow$(1) Given the isomorphism above, condition (3) yields
that the dualising bimodule of $\fibre A{\fp}$ over $k(\fp)$ is
projective, that is to say, $\fibre A{\fp}$ is self-injective.

It remains to verify the last part of the statement. The Gorenstein
property follows from Proposition~\ref{pr:gor=fibre}.  The equivalence
follows from the fact that $\omega_{A/R}$ is projective on both sides
and the isomorphism
\[
A\longiso \Hom_{A}(\omega_{A/R},\omega_{A/R})\,.
\]
See also \cite[Theorem~4.5]{Iyengar/Krause:2022a}.
\end{proof}

\begin{proposition}
\label{pr:IG}
Let $R$ be a Gorenstein ring and $A$ a fibrewise self-injective $R$-algebra. Then 
\[
\injdim_A A = \dim_RA =\injdim_{\op A} {\op A}\,.
\]
In particular, $A$ is Iwanaga--Gorenstein if and only if $\dim_RA$ is finite.
\end{proposition}

\begin{proof}
It suffices to verify the equality for $A$; the one for $A^{\mathrm{op}}$ holds, by symmetry. Since the injective dimension of $A$ is detected by the vanishing of $\Ext_A^i(-,A)$ on finitely generated $A$-modules, it suffices to verify that
\[
\injdim_{A_\fp} A_\fp= \dim_{R_\fp}A_\fp \quad\text{for each $\fp$ in $\supp_RA$.}
\]
We can replace $R$ and $A$ by their localisations at $\fp$ so that $R$ is local, and hence $A$ is semi-local, and $\dim_R A=\dim R$. For any simple $A$-module $L$ one has 
\[
\Ext_A^i(L,\omega_{A/R}) = \Ext_A^i(L,\Hom_R(A,R))\cong \Ext_R^i(A\otimes_RL,R)\,.
\]
Since the ring $R$ is Gorenstein, hence of injective dimension $\dim R$, we deduce that $\Ext_A^i(L,\omega_{A/R})=0$ for $i>\dim R$. Thus Lemma~\ref{le:fibrewise-injective}(3) yields
\[
\Ext_A^i(L,\omega_{A/R})=0 \quad\text{for $i>\dim R$.}
\]
Hence $\injdim_A A\le \dim R$  by \cite[Lemma~B.3.1]{Buchweitz:2021a}. For the converse equality, with $k$ the residue field of $R$, one has  $\Ext_R^d(k,R)\cong k$; see Lemma~\ref{le:regular}. Since $A$ is a non-zero finite free $R$-module, one gets 
\[
\Ext^{d}_A(A\otimes_Rk,A) \cong \Ext_R^{d}(k,A) \ne 0\,.
\]
Thus $\injdim_A A\ge \dim R$. This justifies the stated equalities.
\end{proof}

In the result below, the converse statement need not hold for general Gorenstein algebras, as can be seen by contemplating the case when $R$ is a field.

\begin{lemma}
\label{le:symmetric}
Let $A$ be a Gorenstein $R$-algebra and $M$ an $A$-module. When $M$ is Gorenstein projective, it is Gorenstein projective also as an $R$-module. The converse holds when $A$ is fibrewise self-injective and either $M$ is finitely generated or $\dim_RA$ is finite.
\end{lemma}

\begin{proof}
Since $A$ is finite projective as an $R$-module, any projective $A$-module is also projective as an $R$-module, and hence any acyclic complex of projective $A$-modules is an acyclic complex of projective $R$-modules. It follows that any Gorenstein projective $A$-module is Gorenstein projective also as an $R$-module. 

Suppose that $A$ is fibrewise self-injective and that $M$ is Gorenstein projective as an $R$-module. We verify that $\Ext_A^i(M,-)=0$ for $i\ge 1$ and on $\Prj A$. Given this, if $M$  is finitely generated one can apply \cite[Lemma~6.3]{Iyengar/Krause:2022a} to conclude that it is Gorenstein projective also as an $A$-module. We can draw the same conclusion from \cite[Corollary 11.5.3]{Enochs/Jenda:2000a} for a general $M$ when we also know $\dim_RA$ is finite, for then $A$ is Iwanaga--Gorenstein, by Proposition~\ref{pr:IG}.

As to the vanishing of Ext, any projective $A$-module is a direct summand of a free $A$-module, and any free $A$-module is of the form $A\otimes_RF$ for some free $R$-module $F$. Therefore it suffices to verify that
\[
\Ext^i_A(M,A\otimes_RF)=0\quad\text{for $i\ge 1$.}
\]
Since $A$ is in $\add(\omega_{A/R})$, by Lemma~\ref{le:fibrewise-injective}, the $A$-module $A\otimes_RF$ is in additive subcategory generated by
\[
\omega_{A/R}\otimes_RF  =\Hom_R(A,R)\otimes_RF\cong \Hom_R(A,F)\,.
\]
Thus it suffices to verify that $\Ext_A^i(M,\Hom_R(A,F))=0$ for $i\ge 1$. This follows from  the adjunction isomorphism
\[
\Ext_A^i(M,\Hom_R(A,F)) \cong \Ext_R^i(M,F)
\]
and the hypothesis that $M$ is Gorenstein projective as an $R$-module.
\end{proof}

\section{Cocommutative Hopf algebras}
\label{se:hopf-algebras}
Throughout this section $R$ will be a regular commutative noetherian
ring and $A$ a finite cocommutative Hopf algebra over $R$; this
includes the condition that $A$ is projective as an $R$-module. Then
$\KProj A$ has a natural structure of a tensor-triangulated category,
and one has an analogue of the fibrewise criterion from
Section~\ref{se:local-global} that takes into account this additional
structure. With some further assumption on the cohomology of $A$ we
are then able to stratify the tensor-triangulated category $\KProj A$
via the action of the cohomology ring of $A$.

\subsection*{Tensor structure}
Given $A$-modules $X$ and $Y$, there is a natural diagonal $A$-module structure on $X\otimes_RY$, obtained by restricting its $A\otimes_RA$-module structure along the coalgebra map $\Delta\colon A\to A\otimes_RA$. 

\begin{lemma}
Let $P,Q$ be $A$-modules. If $P$ is projective over $A$ and $Q$ is projective over $R$, then the $A$-module $P\otimes_RQ$, with the usual diagonal action, is projective.
\end{lemma}

\begin{proof}
This follows from the standard adjunction isomorphism 
\[
\Hom_A(P\otimes_RQ,-)\cong \Hom_A(P, \Hom_R(Q,-))\,. \qedhere
\]
\end{proof}

The preceding result implies that $-\otimes_R-$ induces a tensor
product on $\KProj A$. 

\begin{lemma}
\label{le:tensor-unit}
The triangulated category $\KProj A$ is tensor-triangulated, with product $-\otimes_R-$ and unit the
 $A$-complex 
 \[
 \one\colonequals \Hom_R(\pres_{\op A}R,R)\,.
 \]
\end{lemma}

\begin{proof}
We have already seen that $-\otimes_R-$ provides a tensor product  on
$\KProj A$, and it remains to verify the assertion about the unit. Since $\op A$ is noetherian, one can assume that for any $M$ in $\dbcat{\op A}$, its projective resolution $\pres_{\op A}M$ is finitely generated in each degree and  that $(\pres_{\op M})_i=0$ for $i\ll 0$. This fact will be used multiple times in what follows. 

The augmentation $\eps\colon \pres_{\op A}R\to R$ induces the $A$-linear map
\[
\eps^*\colon R\lra \Hom_R(\pres_{\op A}R,R)\,.
\]
For each $X$ in $\KProj A$ this induces the map
\[
X\cong X\otimes_RR \xra{\ X\otimes \eps^*\ } X\otimes_R \Hom_R(\pres_{\op A}R,R)
\]
and the desired result is that this map is an isomorphism. 

Since $\KProj A$ is compactly generated, and the functors involved
preserves coproducts, it suffices to verify the claim when $X$ is
compact, that is to say, of the form $\Hom_A(\pres_{\op A}M,A)$ for some $M$ in $\mathrm{mod}\, \op A$; see~\eqref{eq:kprojc}. Consider the diagram
\[
\begin{tikzcd}
 \Hom_A(\pres_{\op A}M,A) \arrow{drrrr}[swap]{\alpha} \arrow{rrrr}{
   \Hom_A(\pres_{\op A}M,A)\otimes\eps^*}
 	&&&& \Hom_A (\pres_{\op A}M,A) \otimes_R \Hom_R(\pres_{\op A}R,R) \arrow{d}{\cong} \\
 	&&&& \Hom_A(\pres_{\op A}M \otimes_R \pres_{\op A}R,A)
\end{tikzcd}
\]
where the vertical map is the natural one; it is an isomorphism because $\pres_{\op A}A$ and $\pres_{\op A}M$ are degreewise finitely generated. The map $\alpha$ is the obvious composition. It suffices to check that $\alpha$ is an isomorphism.

One can verify that the map $\alpha$ is obtained from the map
\[
\pres_{\op A}M \otimes_R \pres_{\op A}R \xra{\pres_{\op A}M\otimes\eps } \pres_{\op A} M \otimes_R R\cong \pres_{\op A}M
\]
by applying $\Hom_A(-,A)$. Since $\eps$ is a quasi-isomorphism so is
the map above. Since the source and target consist of projective
$\op A$-modules, and are bounded to the right, they are
semi-projective complexes over $\op A$. Thus the map above is a
homotopy equivalence in $\KProj{\op A}$. Therefore applying
$\Hom_A(-,A)$ induces an isomorphism in $\KProj A$. This is the
desired result.
\end{proof}

\subsection*{Rigidity}
The tensor-triangulated category $\KProj A$, which is always compactly
generated, is also rigid when $R$ is regular. We recall briefly the
notion of rigidity in tensor-triangulated categories and refer to
\cite[A.2]{Hovey/Palmieri/Strickland:1997a} for details.

Let $\cat T$ be a compactly generated tensor-triangulated category,
with product $\otimes$ and unit $\one$. We assume that $\one$ is
compact. Being a compactly generated tensor-triangulated category,
$\cat T$ has an \emph{internal function object}, $\fHom(-,-)$, defined
by the property that
\[
\Hom_{\cat T}(X\otimes Y,Z) \cong \Hom_{\cat T}(X,\fHom(Y,Z))
\]
for $X,Y$ and $Z$ in $\cat T$. There is a natural map
\begin{equation}
\label{eq:rigidity}
\fHom(X,\one)\otimes Y \lra \fHom(X,Y)
\end{equation}
and $X$ is \emph{rigid} if this map is an isomorphism for all
$Y$. Since $\one$ is compact, every rigid object is compact, and one
says that $\cat T$ is \emph{rigid} when the converse holds: compact
objects and rigid objects coincide. It is straightforward to verify that for   this property to hold, the following conditions are necessary and sufficient:
\begin{enumerate}[\quad\rm(1)]
\item
the subcategory of compact objects is closed under $\otimes$;
\item
\eqref{eq:rigidity} is an isomorphism when $X,Y$ are compact.
\end{enumerate}

Back to  the tensor-triangulated category $\KProj A$ with product
$\otimes_R$, where $A$ is a cocommutative Hopf algebra over $R$. In
this case, the internal function object can be described quite
concretely as
\[
\fHom(Y,Z) = \bfj \Hom_R(Y,Z)
\]
where $\bfj$ is the right adjoint to the inclusion of $\KProj A$ into the homotopy category of flat $A$-modules; see \cite[Proposition~2.4]{Iyengar/Krause:2006a}. Here is the pertinent result; one can also prove that $\KProj A$ is not rigid when $R$ is not regular.

\begin{lemma}
\label{le:rigidity}
If $R$ is regular, the tensor-triangulated category $\KProj A$ is rigid.
\end{lemma}

\begin{proof}
  As noted above, it suffices to verify that for any $X,Y$ in
  $\KProjc A$, the complex $X\otimes_RY$ is also in $\KProjc A$ and
  that \eqref{eq:rigidity} is an isomorphism. Given the identification
  \eqref{eq:Gorenstein-kprojc} of compact objects in $\KProj A$, this is
  tantamount to verifying that for $M,N$ in $\dbcat A$, the complex
  $M\lotimes_RN$ is in $\dbcat A$ and  the map
\[
\RHom_R(M,R) \lotimes_R N \lra \RHom_R(M,N)
\]
of $A$-complexes is an isomorphism. Both properties are clear, since $R$ is regular.
\end{proof}

\begin{remark}
  The functor $\pi\colon\KProj A\to\KProj{A_k}$ from
  \eqref{eq:adjunction} is a tensor functor which fits -- together with its
  adjoints -- into the framework discussed in  \cite{Balmer/DellAmbrogio/Sanders:2016a}.
\end{remark}

\subsection*{Fibrewise criterion}

Here is the analogue of the fibrewise criterion for detecting membership in localising subcategories, Theorem~\ref{th:lg-kproj}, in the presence of the tensor product. In the statement $\Loc^{\otimes}(Y)$ denotes the tensor ideal localising subcategory generated by $Y$.

\begin{theorem}
\label{th:lg-kproj-ha}
Let $R$ be a regular ring and $A$ a finite cocommutative Hopf $R$-algebra. Let $X,Y$ be objects in $\KProj A$. The following conditions are equivalent. 
\begin{enumerate}[\quad\rm(1)]
\item
 $X\in  \Loc^{\otimes}(Y)$ in $\KProj A$;
\item
 $\fibre X\fp\in \Loc^{\otimes}(\fibre Y{\fp})$ in $\KProj{\fibre A\fp}$ for each $\fp\in \Spec R$.
\end{enumerate}

\end{theorem}

\begin{proof}
  The argument follows the same lines as that for
  Theorem~\ref{th:lg-kproj}, using in addition that $\KProj A$ is a
  rigidly generated tensor-triangulated category by
  Lemma~\ref{le:rigidity}.  Again it is straightforward to verify
  (1)$\Rightarrow$(2), once one observes that the functor $\pi$ from
  \eqref{eq:adjunction} respects the tensor products: For $X,Y$ in
  $\KProj A$, there is a natural isomorphism
\[
\pi(X\otimes_RY) \cong \pi(X) \otimes_{k(\fp)} \pi(Y) \quad\text{in $\KProj{\fibre A{\fp}}$. }
\]

For the implication (2)$\Rightarrow$(1) we use the version of the
local-to-global theorem~\eqref{eq:local-to-global} for
tensor-triangulated categories in
\cite[Theorem~7.2]{Benson/Iyengar/Krause:2011a}. Then the task reduces
to proving for $X,Y$ in $\KProj A$ that when $(R,\fm,k)$ is a local
ring and $X_k$ is in $\Loc^{\otimes}(Y_k)$, the complex $\pi_r\pi(X)$
is in $\Loc^{\otimes}(\pi_r\pi(Y))$. When reducing to the local case,
one uses that for each $\fp$ in $\Spec R$ the category $\KProj{A_\fp}$
identifies with a localising tensor ideal of $\KProj A$ via
\eqref{eq:localisation}.

Whilst the functor $\pi_r$ need not respect tensor products, the
following \emph{projection formula} holds. For $U$ in $\KProj A$ and
$V$ in $\KProj {A_k}$, there is a natural isomorphism
\[
U\otimes_R \pi_r V \cong \pi_r(\pi(U)\otimes_k V)\,.
\]
One can verify this directly, but this is a general fact about tensor
functors and their right adjoints between rigidly generated
tensor-triangulated categories; see
\cite[Theorem~1.3]{Balmer/DellAmbrogio/Sanders:2016a}. This formula
and the fact that, up to direct summands, $\pi$ is surjective on
objects -- see Lemma~\ref{le:unit} -- yield the desired result.
\end{proof}

\subsection*{Finite generation}
Let $R$ be a commutative noetherian ring and $A$ a finite cocommutative Hopf algebra over $R$. Set
\[
S\colonequals \Ext_A^*(R,R)\,.
\]
This is a graded-commutative $R$-algebra. The ring $S$ can be realised as the graded-ring of morphisms
\[
S\cong \Hom^*_{\bfK(A)}(\one,\one)\,.
\]
Since $\one$ is the unit of the tensor product on $\KProj A$,  it has a natural $S$-linear action on it. For each $\fp$ in $\Spec R$ one has that $\fibre A{\fp}$ is a finite dimensional cocommutative Hopf algebra over $k(\fp)$. Set
\begin{equation}
\label{eq:sfp}
S(\fp)\colonequals \Ext_{\fibre A\fp}^*(k(\fp),k(\fp))\,.
\end{equation}
A result of Friedlander and Suslin~\cite{Friedlander/Suslin:1997a} yields that the graded-commutative $k(\fp)$-algebra $S(\fp)$ is finitely generated; equivalently, that it is noetherian.  Throughout we assume that the $R$-algebra $S$ is itself finitely generated. 

\subsection*{Cohomological support}
Let $R$ be a commutative noetherian ring, $A$ a finite cocommutative
Hopf algebra over $R$, and $S$ the cohomology ring introduced
above. We write $\Spec S$ for the homogenous prime ideals in
$S$. Following \cite[\S5]{Benson/Iyengar/Krause:2008a}, the action of $S$
on $\KProj A$ gives rise to a notion of \emph{support} for objects in
$\KProj A$. Namely for each $\fq$ in $\Spec S$ there is a \emph{local
  cohomology functor} functor
$\gam_{\fq}$, defined akin to
\eqref{eq:gammap}. The support of an object $X\in \KProj A$ is by
definition the set
\[
\supp_S X\colonequals \{\fq\in\Spec S\mid \gam_{\fq}X\ne 0\}\,,
\]
and for any class of objects $\cat X$ we set
\[\supp_S\cat X\colonequals\bigcup_{X\in\cat X}\supp_S X\,.\]

\subsection*{Support and fibres}
Theorem~\ref{th:lg-kproj-ha} yields a stratification -- see the
discussion further below -- of $\KProj A$ in terms of subsets of the space
\[
\bigsqcup_{\fp\in\Spec R} \Spec S(\fp)\,,
\]
where to each $X$ in $\KProj A$ we associate the subset 
\[
\bigsqcup_{\fp\in\Spec R} \supp_{S(\fp)} \fibre X{\fp}\,.
\]
The task is to relate this to $\supp_S X$, viewed as a subset of $\Spec S$.

To that end consider the structure map $\eta\colon R\to S$, which induces a map
\[
\eta^{a}\colon \Spec S \lra \Spec R\,.
\]
The fibre of this map over $\fp$ is 
\[
(\eta^{a})^{-1}(\fp) = \Spec(S\otimes_R k(\fp))\,,
\]
which we identify with a subset of $\Spec S$ in the usual way. 

The functor $-\lotimes_R k(\fp)$ induces a map of $R$-algebras
\[
S=\Ext^*_A(R,R) \lra \Ext_{\fibre A\fp}^*(k(\fp),k(\fp))=S(\fp)\,.
\]
This induces the map of graded $k(\fp)$-algebras
\[
\kappa_{\fp}\colon S\otimes_R k(\fp)  \lra  S(\fp)\,.
\]
Even in the best of cases, one does not expect this to be an isomorphism. Consider the induced map on spectra:
\begin{equation}
\label{eq:fibre-map}
\kappa^a_{\fp}\colon \Spec S(\fp) \lra \Spec (S\otimes_R k(\fp))\subseteq \Spec S\,.
\end{equation}

The result below tracks the behavior of supports as we pass to the fibres. 

\begin{lemma}
\label{le:support-fibre}
Let $R$ be a regular ring and  $A$ a finite cocommutative Hopf $R$-algebra. Fix $\fp$ in $\Spec R$ and let $\kappa^a_{\fp}$ be the map in \eqref{eq:fibre-map}. For each $X$ in $\KProj A$  there is an equality
\[
\kappa^a_{\fp}(\supp_{S(\fp)}\fibre X{\fp}) = \supp_SX \cap (\eta^{a})^{-1}(\fp)\,.
\]
\end{lemma}

\begin{proof}
  We can reduce to the case where $(R,\fm,k)$ is a local ring and
  $\fp=\fm$, the maximal ideal of $R$. Set
  $S_k\colonequals S\otimes_Rk =S/\fm S$. Via the map $\kappa_\fm\colon S_k\to S(\fm)$
the $S(\fm)$-action on $\KProj{A_k}$ induces an $S_k$-action. Applying \cite[Corollary~7.8(1)]{Benson/Iyengar/Krause:2012b} to the identity functor on $\KProj {A_k}$ yields an equality 
\[
\kappa^a_{\fm}(\supp_{S(\fm)}X_k) = \supp_{S_k} X_k\,.
\]
It thus suffices to work with the subset on the right.

Consider the adjunction \eqref{eq:adjunction}.  For any compact object $C$ in $\KProj A$ one has isomorphisms of graded $S_k$-modules
\begin{align*}
\Hom_{\mathbf K(A_k)}^*(\pi C,X_k) 
	&= \Hom_{\mathbf K(A_k)}^*(\pi C,\pi X)  \\
	& \cong \Hom_{\mathbf K(A)}^*(C,\pi_r\pi X) \\
	& \cong \Hom_{\mathbf K(A)}^*(C, X\otimes_RK)
\end{align*}
where the last isomorphism is by Lemma~\ref{le:unit}. Any compact
object in $\KProj{A_k}$ is a direct summand of $\pi C$ for some compact object $C$ in $\KProj A$
by Lemma~\ref{le:unit}. By \cite[Theorem~5.2]{Benson/Iyengar/Krause:2008a}, one can compute $\supp_{S_k}(X_k)$ from the
support of the $S_k$-modules $\Hom_{\mathbf K(A_k)}^*(\pi C,X_k)$. This gives the first equality below
\begin{align*}
\supp_{S_k} X_k 
	&=\supp_{S}(X\otimes_RK) \\
	&= \supp_{S}X\cap V(\fm S) \\
	&=\supp_{S}X\cap (\eta^a)^{-1}(\fm)\,.
\end{align*}
The second one holds because $X\otimes_RK$ represents $\kos X{\fm  S}$; see  \cite[Lemma~2.6]{Benson/Iyengar/Krause:2011a}.
\end{proof}

\subsection*{Stratification}
Let $S$ be a graded commutative noetherian ring and $\cat T$ a rigidly
compactly generated tensor-triangulated category. We denote by
$\cat T^{\mathrm c}$ the full subcategory of compact objects.  We say
that $\cat T$ is \emph{$S$-linear} to mean that $S$ acts on $\cat T$
via a map of graded rings
\[
S\lra \mathrm{End}^*_{\cat T}(\one)\,;
\]
see \cite[\S7]{Benson/Iyengar/Krause:2011a}. We are mainly interested
in the case 
\[\cat T=\KProj A\qquad\text{with}\qquad \cat T^{\mathrm c}\cong\dbcat A\] for a cocommutative Hopf algebra $A$ over $R$, and $S$ its cohomology algebra as above, but it is convenient to make a few observations in greater generality. 

In the context above, one says that the tensor-triangulated category
$\cat T$ is \emph{stratified} by $S$ if for each $\fq$ in $\Spec S$
the category $\gam_{\fq}\cat T$, consisting of the $\fq$-local and
$\fq$-torsion objects in $\cat T$, is either zero or minimal, in that,
it has no proper localising tensor ideals; see
\cite[\S7.2]{Benson/Iyengar/Krause:2011a}. When this holds one has
 a bijection
\begin{equation}
\label{eq:classification}
\left\{
\begin{gathered}
  \text{Localising tensor ideals of $\cat T$}
\end{gathered}
\right\} \xra{\ \supp_S(-)\ } \{
\begin{gathered}
  \text{Subsets of $\supp_{S} \one$}
\end{gathered}
\}\,.
\end{equation}
When in addition the graded $S$-module $\Ext^*_A(M,M)$ is finitely generated for each $M\in \rmod A$,  the subset $\supp_SM$ of $\Spec S$ is closed and, by \cite[Theorem~6.1]{Benson/Iyengar/Krause:2011a}, the bijection above restricts to a bijection
\begin{equation}
\label{eq:thick-classification}
\left\{
\begin{gathered}
  \text{Thick tensor ideals of ${\cat T}^{\mathrm c}$}
\end{gathered}
\right\} \xra{\ \supp_S(-)\ } \left\{
\begin{gathered}
  \text{Specialisation closed}\\ \text{subsets of $\supp_{S} \one$}
\end{gathered}
\right\}\,.
\end{equation}

Here is a slightly different perspective on the stratification property.

\begin{lemma}
\label{le:stratification}
Let $\cat T$ be an $S$-linear tensor-triangulated category as above. Then the tensor-triangulated category $\cat T$ is stratified by $S$ if and only if for any $X,Y$ in $\cat T$ there are equivalences
\begin{align*}
\supp_SX\subseteq \supp_SY 
	&\iff \Loc^{\otimes}(X) \subseteq \Loc^{\otimes}(Y) \\
 	&\iff X\in \Loc^{\otimes}(Y)\,.
\end{align*} 
\end{lemma}

\begin{proof}
  We use the fact that for each $\fq$ in $\Spec S$ we have
  \[\gam_\fq\cat T=\{X\in\cat T\mid\supp_S X \subseteq\{\fq\}\}\]
  by \cite[Corollary~5.9]{Benson/Iyengar/Krause:2008a}.
Evidently when the stated property holds the localising tensor ideal $\gam_{\fq}\cat T$ of $\cat T$ is minimal for each $\fq$ in $\Spec S$, so $\cat T$ is stratified by $S$ as a tensor-triangulated category. The converse is equally clear.
\end{proof}

Here is one of the main results of our work. When it applies, one gets
a classification of the localising tensor ideals of $\KProj A$ and
also the thick tensor ideals of its subcategory of compact objects,
which identifies with $\dbcat A$.

\begin{theorem}
\label{th:homeo=stratification}
Let $R$ be a regular ring and $A$ a finite cocommutative Hopf algebra
over $R$ such that the $R$-algebra $S\colonequals \Ext^*_A(R,R)$ is finitely generated. If the map $\kappa^a_\fp$ in \eqref{eq:fibre-map} is bijective for each $\fp$ in $\Spec R$, then the tensor-triangulated category $\KProj A$ is stratified by the action of $S$, and $\supp_S \KProj A = \Spec S$.
\end{theorem}

\begin{proof}
  The main task is to verify that when $X,Y$ are objects in $\KProj A$
  with $\supp_SX\subseteq \supp_SY$, the complex $X$ is in
  $\Loc^{\otimes}(Y)$; see Lemma~\ref{le:stratification}. Since
  $\kappa_\fp^a$ is a homeomorphism for each $\fp$ in
  $\Spec R$, Lemma~\ref{le:support-fibre} yields an inclusion
\[
\supp_{S(\fp)}\fibre X{\fp}\subseteq \supp_{S(\fp)}\fibre Y{\fp}\,.
\]
Since $\fibre A{\fp}$ is a finite dimensional cocommutative Hopf algebra over $k(\fp)$, the triangulated category $\KProj{\fibre A\fp}$ is stratified by the action of it cohomology algebra, $S(\fp)$; this is the main result of \cite{Benson/Iyengar/Krause/Pevtsova:2018a}. Thus the inclusion above implies
\[
\fibre X{\fp}\in \Loc^{\otimes}(\fibre Y{\fp})\,.
\]
This holds for each $\fp$ in $\Spec R$, so we can apply Theorem~\ref{th:lg-kproj-ha}  to deduce that $X$ is in $\Loc^{\otimes}(Y)$ as desired.

It remains to observe that $\supp_S \one= \Spec S$, as follows, from example, from Lemma~\ref{le:support-fibre}, for $\fibre {\one}{\fp}$ is the unit of $\KProj {\fibre A{\fp}}$ and its support is $\Spec S(\fp)$.
\end{proof}

In Section~\ref{se:finite-groups} we prove that group algebras of
finite groups satisfy the hypotheses of the preceding result. Here is one more family of examples to which it applies.

\begin{example}
\label{ex:exterior}
Let $R$ be a regular ring.  Set $A\colonequals \wedge_R F$, the exterior algebra on a finite free $R$-module $F$. We view it as 
a $\bbZ/2$-graded Hopf algebra, with coalgebra structure defined by $\Delta(x)=x\otimes 1+ 1\otimes x$ for $x\in F$. In this case $\Ext^*_A(R,R)$ is the symmetric algebra on $\Hom_R(F,R)$. Given this it is clear that the hypotheses of Theorem~\ref{th:homeo=stratification} are satisfied in this case. 

Here is another family of examples: Suppose $k$ is field, $R$ a $k$-algebra, and that the Hopf algebra $A$ is of the form $R\otimes_k A'$ where $A'$ is a finite dimensional cocommutative Hopf algebra over $k$. Then $\Ext^*_A(R,R)\cong R\otimes_k \Ext_{A'}(k,k)$ as graded $R$-algebras. With this, it is easy to verify that $A$ falls under the purview of Theorem~\ref{th:homeo=stratification}.
\end{example}

Next we prepare to prove Theorem~\ref{ith:homeo=stratification} stated
in the introduction. 

\subsection*{Gorenstein projective modules}
Let $R$ be a regular ring and $A$ a finite cocommutative Hopf
$R$-algebra. The fibres $\fibre A{\fp}$ are finite dimensional
cocommutative Hopf algebras over $k(\fp)$, hence
self-injective~\cite[Lemma~I.8.7]{Jantzen:2003a}. Thus the $R$-algebra
$A$ is fibrewise self-injective and therefore Gorenstein, by
Proposition~\ref{pr:gor=fibre}. As $R$ is regular, Gorenstein
projective $R$-modules are projective by Lemma~\ref{le:regular}. Hence
a Gorenstein projective $A$-module is projective as an $R$-module, and
the converse holds if the module is finitely generated or $\dim R$ is
finite; see Lemma~\ref{le:symmetric}.

\begin{lemma}\label{le:GProj-Kac}
Let $R$ be a regular ring and $A$ a finite cocommutative Hopf algebra. The tensor product $-\otimes_R-$ with diagonal $A$-action endows $\uGProj A$ with a structure of a rigidly compactly generated tensor-triangulated category. This structure is compatible with the equivalence in Theorem~\ref{th:Buchweitz-big}.
\end{lemma}

\begin{proof}
  If $X$ is an acyclic complex of projective $A$-modules and $N$ is a
  Gorenstein projective $A$-module, then the complex $X\otimes_RN$ of
  projective modules is also acyclic, for $N$ is projective as an
  $R$-module. It follows that if $M$ is a Gorenstein projective
  $A$-module, so is $M\otimes_RN$. Thus the category of Gorenstein
  projective $A$-modules is closed under $-\otimes_R-$, and $R$,
  viewed as an $A$-module via the augmentation $A\to R$ is the unit of
  this product. Observe that as an $A$-module $R$ is Gorenstein
  projective, for it is finitely generated, and projective as an
  $R$-module.  Since the $A$-module $P\otimes_RN$ is projective when
  $P$ is projective, this tensor product induces one on the stable
  category, $\uGProj A$, making it a tensor-triangulated category,
  with unit $R$. The function object on $\uGProj A$ is $\Hom_R(-,-)$,
  and given this it is easy to verify the rigid objects in it are precisely
  the compact objects, that is to say, isomorphism classes of the
  finitely generated Gorenstein projective modules. In summary,
  $\uGProj A$ is rigidly compactly generated.

  A straightforward computation shows that the assignment in
  Theorem~\ref{th:Buchweitz-big} is compatible with the tensor
  structures.
\end{proof}

Let $S\colonequals \Ext^*_A(R,R)$ be the cohomology algebra as
before. We write $\Prj S$ for the projective spectrum of $S$, namely,
those prime ideals in $\Spec S$ that do not contain the ideal
$S^{\geqslant 1}$ of positive degree elements.

\begin{lemma}\label{le:Kac-tensor-ideal}
  With the assumptions from Theorem~\ref{th:homeo=stratification}, the
  full subcategory  $\KacProj A$ of  $\KProj A$ is a localising tensor
  ideal with support $\Prj S$.
\end{lemma}
\begin{proof}
  For each $\fp$ in $\Spec R$ the functor $X\mapsto X_{k(\fp)}$
  maps the recollement \eqref{eq:recollement} for $A$ to the
  corresponding recollement for $A_{k(\fp)}$. To see this, observe
  that the recollement \eqref{eq:recollement} is determined by
  functorial exact triangles $\pres X\to X\to \bft X\to$ for each $X$
  in $\KProj A$ such $\pres X$ is K-projective and $\bft X$ is
  acyclic. These properties are preserved by $(-)_{k(\fp)}$ since the
  functor is exact and preserves all coproducts. For K-projectives
  this is clear, since they are generated by perfect complexes which
  are preserved by $(-)_{k(\fp)}$. For acyclic complexes, see
  Lemma~\ref{le:gproj-fibre}.

  The assertion of the lemma now follows since $\KacProj {A_{k(\fp)}}$ is a
  localising tensor ideal of $\KProj {A_{k(\fp)}}$ with support
  $\Prj S(\fp)$; see \cite[\S10]{Benson/Iyengar/Krause/Pevtsova:2018a}.
\end{proof}

We are now ready to prove our stratification result for
representations of finite cocommutative Hopf algebras.

\begin{proof}[Proof of Theorem~\ref{ith:homeo=stratification}]
  We use the triangle equivalence $\uGProj A\iso \KacProj A$ from
  Theorem~\ref{th:Buchweitz-big}, which preserves the tensor structure
  and the $S$-action thanks to Lemma~\ref{le:GProj-Kac}.  Now the
  stratification of $\uGProj A$ via $S$ follows from
  Theorem~\ref{th:homeo=stratification}, since $\KacProj A$ is a
  localising tensor ideal of $\KProj A$ by
  Lemma~\ref{le:Kac-tensor-ideal}.  In particular, the support of
  $\uGProj A$ is precisely $\Prj S$.
\end{proof}

\section{Finite groups}
\label{se:finite-groups}

Let $G$ be a finite group. The main result of this section is that for $A=RG$, the group algebra of $G$ over any commutative noetherian ring $R$, the map \eqref{eq:fibre-map} is a homeomorphism for all primes in the spectrum of $R$. As a consequence we get a stratification theorem when $R$ is regular; see Theorem~\ref{th:stratification-RG}. In this case the cohomology algebra $\Ext^*_{RG}(R,R)$ is the group cohomology algebra. This is usually denoted $H^*(G,R)$, and we follow suit.

\subsection*{Cup products}
Let $R$ be a commutative ring, not necessarily noetherian,  and $M$ an $R$-module, both viewed  as $G$-modules with trivial action. The cup product makes $H^*(G,R)$  an $R$-algebra and $H^*(G,M)$  a module over $H^*(G,R)$. These are defined as follows: Let $P$ be a projective resolution of the trivial $\bbZ G$-module $\bbZ$ and $\Delta\colon P\to P\otimes_{\bbZ}P$ a diagonal approximation.  Given classes $x\in H^*(G,R)$ and $y\in H^*(G,M)$, represented by cocycles $\tilde x \in \Hom_{\bbZ G}(P,R)$ and $\tilde y\in \Hom_{\bbZ G}(P,M)$ the cup product $x \cup y$ is represented by the class of the composition of maps
\[
P \xrightarrow{\ \Delta\ } P\otimes_{\bbZ} P \xrightarrow{\ \tilde x \otimes \tilde y\ } R \otimes_{\bbZ} M \lra M\,.
\]
If $I\subseteq R$ is an ideal, $-\cup-$ defines a product on $H^*(G,I)$. It is clear from the definition that if $I$ is nilpotent of order $n$, then so is $H^*(G,I)$.

\subsection*{Infinitesimal deformations of coefficients}
Let $R\to R'$ be a  surjective map of commutative rings whose kernel, say $I$, satisfies $I^2=0$;  thus $I$ is an $R'$-module. One thinks of $R$ as  an infinitesimal deformation of $R'$. The exact sequence 
\begin{equation}
\label{eq:deformation}
0\lra I \lra R\lra R'\lra 0
\end{equation}
induces a connecting homomorphism
\[
\delta\colon H^{*}(G,R') \to H^*(G,I)\,.
\]
Since $I$ is an $R'$-module $H^*(G,I)$ is a module over $H^*(G,R')$, via the cup product.

The statement of the result below, and its proof, are a variation on \cite[Lemma~4.3.3]{Benson:1998c}.

\begin{lemma}
\label{le:derivation}
In the context above, for $x,y\in H^*(G,R')$  one has
\[
\delta(x\cup y) = \delta(x)\cup y + (-1)^{|x|}x\cup \delta(y)\,.
\]
\end{lemma}

\begin{proof}
As in the proof of \cite[Lemma~4.3.3]{Benson:1998c}, let $P$ be the projective resolution of the trivial $\bbZ G$-module $\bbZ$, and
$\Delta \colon P \to P \otimes_{\bbZ} P$  a diagonal approximation. The exact sequence~\eqref{eq:deformation} induces the exact sequence of complexes
\[ 
0 \lra \Hom_{\bbZ G}(P,I) \lra \Hom_{\bbZ G}(P,R) \lra \Hom_{\bbZ G}(P,R') \lra 0\,. 
\]
Represent $x$ and $y$ by cocycles $\tilde x$ and $\tilde y$ on the right hand side of this sequence. Then $x \cup y$ is represented by
the composite 
\[
\tilde x \cup \tilde y \colon P \xrightarrow{\ \Delta\ } P\otimes_{\bbZ} P \xrightarrow{\ \tilde x \otimes \tilde y\ } R' \otimes_{\bbZ} R'
\xrightarrow{\ \mu\ } R'
\]
where $\mu$ is the multiplication map. To compute the effect of the connecting homomorphism, we first lift $\tilde x$ and $\tilde y$ to
cochains $\hat x$ and $\hat y$ in  $\Hom_{\bbZ G}(P,R)$. Since $d\tilde x=0 = d\tilde y$, the elements $d\hat x$ and $d\hat y$ lie in $I$. 
The element $\tilde x\cup \tilde y$ lifts to $\hat x\cup \hat y$, and
\begin{align*}
d(\hat x \cup \hat y) 
&=d\hat x \cup \hat y + (-1)^{|x|}\hat x \cup d\hat y \\
&=d\hat x \cup \tilde y + (-1)^{|x|}\tilde x \cup d\hat y
\end{align*}
The second equality holds as $d\hat x$ and $d\hat y$ lie in $I$. This gives the stated equality.
\end{proof}

In what follows, we say that an abelian group $M$ is \emph{$p$-local}, for a prime number $p$, if the natural map $M\to M_{(p)}$ is an isomorphism. 

\begin{lemma}
\label{le:F-onto}
Let $p$ be a prime dividing $|G|$ and $\pi\colon R\to R'$ a map  of $p$-local rings,  with $\Ker(\pi)$ nilpotent. Then the map $H^*(G,\pi)$ has nilpotent kernel, and there exists an integer $n$ such that for any element $x\in  H^{\geqslant 1}(G,R')$, the element $x^{p^n}$ is in the image of the map $H^*(G,\pi)$. 
\end{lemma}

\begin{proof}
Set $I\colonequals \Ker(\pi)$. Since this ideal is nilpotent, so is $H^*(G,I)$, under cup products.  The claim about nilpotence is clear because, by the exact sequence in cohomology arising from \eqref{eq:deformation}, the kernel of $H^*(G,\pi)$ is the image of the map 
\[
H^*(G,I) \lra H^*(G,R)
\]
which respects cup products.

As to the second part of the statement,  it suffices to consider the case where $I^2=0$.  Let $n$ be the largest integer such that $p^n$ divides $|G|$.  Since $|G|$ annihilates $H^{\geqslant 1}(G,R')$, and the ring $R'$ and hence also $H^*(G,R')$ is $p$-local, one gets that
\[
p^n\cdot H^{\geqslant 1}(G,R')=0\,.
\]
If $|x|$ is odd, then $2\cdot x^2=0$, since $H^*(G,R')$ is graded-commutative. Thus if also $p$ is odd, then $x^2=0$, since we are in the $p$-local situation. Thus we can suppose either $|x|$ is even or $p=2$. In either case, a repeated application of Lemma~\ref{le:derivation} yields $\delta(x^i) = i x^{i-1}\delta(x)$ for each $i\ge 1$. In particular $\delta(x^{p^n})=0$. It then follows from the exact sequence in group cohomology arising from \eqref{eq:deformation} that $x^{p^n}$ is in the image of the map $H^*(G,R)\to H^*(G,R')$.
\end{proof}

\subsection*{Modules with bounded torsion}
Let $M$ is an abelian group such that its torsion-subgroup, denoted $\tors(M)$ is bounded; that is to say, there exists an integer $n$ such that $n\cdot \tors(M)=0$. Fomin~\cite{Fomin:1937a} proved that inclusion $\tors(M)\subseteq M$  splits; see also \cite[Corollary pp. 134]{Kaplansky:1952a}. This result will be used below.

\begin{lemma}
\label{le:bounded-torsion}
Let $p$ be a prime dividing $|G|$ and $M$ a $p$-local abelian group such that $\tors(M)$ is bounded. For all integers $s\gg 0$ the map
\[
H^{\geqslant 1}(G,M) \lra H^{\geqslant 1}(G,M/p^sM)
\]
induced by the surjection $M\to M/p^sM$, is one-to-one.
\end{lemma}

\begin{proof}
Since $M$ is $p$-local, the only torsion is $p$-torsion. Choose $s\gg 0$ such that 
\[
p^s \tors(M)=0= p^s H^{\geqslant 1}(G,M)\,.
\]
The equality on the left means that the sequence below, where the map $M\to p^sM$ is given by $m\mapsto p^sm$, is exact:
\[
0\lra \tors(M) \lra M \lra p^s M\lra 0\,.
\]
This is split-exact, by Fomin's result recalled above, so the induced map  
\[
H^*(G,M) \lra H^*(G,p^sM)
\]
is surjective.  The map $M\xra{ p^s }M$ factors as $M\to p^sM \to M$ where the one on the right is inclusion. By the choice of $p^s$, the composition of the induced maps
\[
H^*(G,M) \lra H^*(G,p^sM) \lra H^*(G,M)
\]
is zero in degrees $\geq 1$. Since the map on the left is surjective, it follows that the one on the right is zero in degrees $\geq 1$. Then the cohomology exact sequence arising from the exact sequence
\[
0\lra p^s M \lra M \lra M/p^sM \lra 0
\]
yields the desired statement. 
\end{proof}

\subsection*{Noetherian ring of coefficients}
The result below was proved by  Benson and
Habegger~\cite{Benson/Habegger:1987a} when $R= \bbZ$; the argument
given here is modeled on their proof.  A general result, allowing non-trivial $G$-action on $R$, was proved by Lau~\cite[Section~7]{Lau:Bsp}. 

Recall that a map of rings $f\colon S\to T$ containing a field of positive characteristic $p$ is an \emph{F-isomorphism} if $\ker(f)$ is nilpotent, and for each $t\in T$ there exists an $n$ such that $t^{p^n}$ is the image of $f$. 

\begin{theorem}
\label{th:f-iso-RG}
Let $G$ be a finite group and $R$ a  commutative noetherian  ring. For each prime number $p$, the map
\[ 
H^*(G,R) \otimes_R R/pR \lra H^*(G,R/pR) 
\]
is an F-isomorphism.
\end{theorem}

\begin{proof}
Set $R_{(p)}\colonequals \bbZ_{(p)}\otimes_{\bbZ} R$. As $R/pR$ is $p$-local, the map $R\to R/pR$ factors  through 
$R_{(p)}$. As localisation is an exact functor, there are natural isomorphisms
\begin{gather*}
H^*(G,R)\otimes_R R/pR \cong H^*(G,R_{(p)})\otimes_{R_{(p)}} R/pR \\
H^*(G,R/pR) \cong H^*(G,R_{(p)}/pR_{(p)})\,.
\end{gather*}
Thus replacing $R$ by $R_{(p)}$ we can assume $R$ is $p$-local.

For any finitely generated $R$-module $M$, the (additive) torsion submodule $\tors(M)$ is an $R$-submodule of $M$, and hence finitely generated as an $R$-module,  as $R$ is noetherian. It follows that $\tors(M)$ is bounded, so Lemma~\ref{le:bounded-torsion} applies. 

Choose an integer $s$ large enough that the conclusion of \emph{op.~cit.} applies to the $R$-modules $R$ and $pR$. Consider the commutative diagram of coefficients
\[ 
\begin{tikzcd}
0 \arrow{r} & pR \arrow{d}{\theta} \arrow{r}
	& R \arrow{d}{\pi} \arrow{r} & R/pR \arrow[equal]{d}\arrow{r} & 0 \\
0 \arrow{r} & pR/p^{s+1}R \arrow{r} & R/p^{s+1}R 
\arrow{r} & R/pR \arrow{r} & 0 
\end{tikzcd} 
\]
This induces a commutative diagram
\begin{equation*} 
\begin{tikzcd}
 H^*(G,R) \arrow{d}{\pi^*} \arrow{r}{\iota_1} & H^*(G,R/pR)
\arrow[equal]{d}\arrow{r}{\delta_1} & H^*(G, pR)
\arrow{d}{\theta^*} \\
H^*(G,R/p^{s+1}R) \arrow{r}{\iota_2} & H^*(G,R/pR) \arrow{r}{\delta_2} & H^*(G, pR/p^{s+1}R)
\end{tikzcd} 
\end{equation*}
where $\delta_1$ and $\delta_2$ are the connecting maps. The choice of $s$ ensures that $\pi^*$ and $\theta^*$ are injective in positive degrees; see Lemma~\ref{le:bounded-torsion}. Since $\pi^{\geqslant 1}$ is injective and the kernel of $\iota_2$ is nilpotent, by Lemma~\ref{le:F-onto}, so is the kernel of $\iota_1$ in positive degrees. This map factors through the map
\[
H^*(G,R)\otimes_R R/pR \lra H^*(G,R/pR)
\]
so the latter is one-to-one up to nilpotence.

Fix $x\in H^{\geqslant 1}(G,R/pR)$. Applying  Lemma~\ref{le:F-onto} to the map $R/p^sR\to R/pR$ yields that for some $n\ge 1$ the element $x^{p^{n}}$ is in the image of $\iota_2$.  So in the diagram above, we have $\delta_2(x^{p^n})=0$. Since $\theta^{\geqslant 1}$ is injective we have $\delta_1(x^{p^n})=0$. It follows from the exactness of the top row of the diagram that $x^{p^n}$ is in the image of  $\iota_1$. This implies that the map in the statement of the theorem is F-onto.
\end{proof}

Here is a consequence of the preceding theorem. 

\begin{corollary}
\label{co:f-iso-RG}
Let $G$ be a finite group and $R$ a  commutative noetherian  ring.  For any map of rings  $R\to k$ with $k$ a field of positive characteristic, the natural map
\[
\hh{G,R}\otimes_Rk \lra \hh{G,k}
\]
is an $F$-isomorphism, and hence the induced map on spectra 
\[
\Spec \hh{G,k} \lra \Spec (\hh{G,R}\otimes_Rk)
\]
is a  homeomorphism.
\end{corollary}

\begin{proof}
Let $p$ be the characteristic of $k$. The map $R\to k$ factors through $R/pR$. Applying $-\otimes_{R/pR}k$ to the F-isomorphism in Theorem~\ref{th:f-iso-RG} yields the F-isomorphism
\[
H^*(G,R)\otimes_R k \lra H^*(G,R/pR)\otimes_{R/pR} k\,.
\]
It remains to observe that the right hand side is isomorphic to $H^*(G,k)$.
\end{proof}

\begin{theorem}
\label{th:stratification-RG}
With $G$ a finite group and $R$ a regular ring, the tensor-triangulated category $\KProj {RG}$ is stratified by the action of $H^*(G,R)$.
\end{theorem}

\begin{proof}
  The $R$-algebra $H^*(G,R)$ is finitely generated by a result of
  Evens \cite{Evens:1961a} and Venkov \cite{Venkov:1959a}. Thus the
   result follows from
  Theorem~\ref{th:homeo=stratification} and Corollary~\ref{co:f-iso-RG}.
\end{proof}

\begin{chunk}
  Theorem~\ref{th:stratification-RG} yields classifications of thick
  and localising tensor ideals which have recent predecessors.  The
  classification of thick tensor ideals of $\dbcat {RG}$ has been
  obtained by Lau~\cite{Lau:Bsp} from a geometric perspective
  involving Deligne--Mumford stacks. Using homotopy
  theoretic methods as well as the result of Lau,  Barthel
  \cite{Barthel:strat, Barthel:strat-regular} classifies  the localising tensor ideals of $RG$-modules that are projective as $R$-modules.
\end{chunk}

\begin{chunk}
  For a group algebra $RG$ there are other possible versions of a
  stable module category. For instance, one can endow the category of
  $RG$-modules with the exact structure given by those short exact
  sequences which are split exact when restricted to the trivial
  subgroup. Then the category $\rmod RG$ of finitely presented
  $RG$-modules is a Frobenius category, and the corresponding stable
  category $\stmod RG$ is in fact a tensor triangulated category
  \cite{Benson/Iyengar/Krause:2013a}. For general $R$ this category is
  different from $\uGproj RG$, as can be seen from the discussion in
  \cite[\S7]{Benson/Iyengar/Krause:2013a} for $R=\bbZ$.
\end{chunk}

\end{document}